\documentclass[a4paper,10pt]{amsart}
\usepackage{amsfonts}
\usepackage{amssymb}
\usepackage[utf8]{inputenc}
\usepackage{amsmath}
\usepackage{pdflscape}
\usepackage{float}
\usepackage{easyReview}
\usepackage{graphicx}
\usepackage[colorlinks=true, linkcolor=red, citecolor=blue]{hyperref}
\usepackage[]{epsfig}
\usepackage[]{pstricks}
\usepackage{tikz}

\newtheorem{theorem}{Theorem}[section]
\newtheorem{proposition}[theorem]{Proposition}
\newtheorem{lemma}[theorem]{Lemma}

\newtheorem{remark}{Remark}
\theoremstyle{remark}

\newtheorem{definition}{Definition}
\usepackage{mathrsfs}
\def\supp{\operatorname{{supp}}}

\numberwithin{equation}{section}
\makeatletter
\@namedef{subjclassname@2010}{\textup{2020} Mathematics Subject Classification}
\makeatother

\begin{document}

%journals
% Annales de l’Institut Fourier A2 
% Discrete and Continuous Dynamical Systems (DCDS) A2
% JOURNAL D'ANALYSE MATHÉMATIQUE (JERUSALEM) A2
%Communications on Pure and Applied Analysis A2

%REVISTA DE LA REAL ACADEMIA DE CIENCIAS EXACTAS, FÍSICAS Y NATURALES. SERIE A, MATEMÁTICAS (ED. IMPRESA) A3
%JOURNAL OF MATHEMATICAL PHYSICS A4

% Nonlinear Analysis A2
%BULLETIN BRAZILIAN MATHEMATICAL SOCIETY A4

	%The flatness approach for the higher-order KdV equation
	\title[Flatness approach for 5th KdV equation]{Control of Kawahara equation using flat outputs}
	\author[Capistrano--Filho]{Roberto de A. Capistrano--Filho$^*$}
	\email{roberto.capistranofilho@ufpe.br}
	\email{jandeilson.santos@ufpe.br}
		\author[da Silva]{Jandeilson Santos da Silva}
	\address{Departamento de Matem\'atica, Universidade Federal de Pernambuco, S/N Cidade Universit\'aria, 50740-545, Recife (PE), Brazil}
	
		\thanks{$^*$Corresponding author: roberto.capistranofilho@ufpe.br}
	\thanks{Capistrano–Filho was partially supported by CAPES/PRINT grant number 88887.311962/2018-00, CAPES/COFECUB grant number 88887.879175/2023-00, CNPq grant numbers 421573/2023-6, 307808/2021-1 and 401003/2022-1 and Propesqi (UFPE). Da Silva acknowledges support from CNPq.}

\begin{abstract}
In this study we focused on the linear Kawahara equation in a bounded domain, employing two boundary controls. The controllability of this system has been previously demonstrated over the past decade using the Hilbert uniqueness method which involves proving an observability inequality, in general, demonstrated via Carleman estimates. Here, we extend this understanding by achieving the exact controllability within a space of analytic functions, employing the flatness approach which is a new approach for higher-order dispersive systems.	
\end{abstract}

%\begin{abstract}
%  \setlength{\parindent}{0pt}% Remove paragraph indentation
% % \textbf{Abstract A}
% \par
%In this study we focused on the linear Kawahara equation in a bounded domain, employing two boundary controls. The controllability of this system has been previously demonstrated over the past decade using the Hilbert uniqueness method which involves proving an observability inequality, in general, demonstrated via Carleman estimates. Here, we extend this understanding by achieving the exact controllability within a space of analytic functions, employing the flatness approach which is a new approach for higher-order dispersive systems.	
%  \par
%  \medskip % Add a small space between the two abstracts
%\textsc{R\'esum\'e.}\ Dans cette étude, nous nous sommes concentrés sur l'équation linéaire de Kawahara dans un domaine borné, en utilisant deux contrôles aux frontières. La contrôlabilité de ce système a déjà été démontrée au cours de la dernière décennie en utilisant la méthode d'unicité de Hilbert qui consiste à prouver une inégalité d'observabilité, en général démontrée via les estimations de Carleman. Ici, nous étendons cette compréhension en obtenant la contrôlabilité exacte dans un espace de fonctions analytiques, en utilisant ``flatness approach" qui est une nouvelle approche pour les systèmes dispersifs d'ordre supérieur.
%\end{abstract}

	\date{\today}
	\maketitle
	
	%\tableofcontents
	
	\thispagestyle{empty}
	
	%***********************************************
	%\normalsize
\section{Introduction}
\subsection{Background and literature review}
Many scientists from diverse fields, including hydraulic engineering, fluid mechanics, physics, and mathematics, have studied water wave models. These models are generally difficult to derive and complex to analyze for qualitative information about wave dynamics, making their study interesting and challenging. Recently, researchers have used appropriate assumptions on amplitude, wavelength, wave steepness, and other factors to investigate asymptotic models for water waves and gain insights into the full water wave system. For rigorous justification of various asymptotic models for surface and internal waves, see, for instance, \cite{ASL,BLS,Lannes} and references therein.

In an appropriate non-dimensional form, water waves can be modeled as a free boundary problem of the incompressible, irrotational Euler equation. This involves two non-dimensional parameters: $\delta := \frac{h}{\lambda}$ and $\varepsilon := \frac{a}{h}$, where the water depth, the wavelength, and the amplitude of the free surface are respectively denoted by $h, \lambda$ and $a$. Additionally, the parameter $\mu$, known as the Bond number, measures the relative importance of gravitational forces compared to surface tension forces. Long waves, also known as shallow water waves, are characterized by the condition $\delta \ll 1$. There are various long-wave approximations depending on the relationship between  $\varepsilon$ and $\delta$.

The discussion above suggests that, instead of relying on models with poor asymptotic properties, one can rescale the mentioned parameters to find systems that reveal asymptotic models for surface and internal waves, such as the Kawahara model. Specifically, by setting $\varepsilon = \delta^4 \ll 1$, $\mu = \frac13 + \nu\varepsilon^{\frac12}$, considering the critical Bond numbe $\mu = \frac13$, the Kawahara equation is derived. This equation, first introduced by Hasimoto and Kawahara \cite{Hasimoto1970, Kawahara}, takes the form
$$\pm2 v_t + 3vv_x - \nu v_{xxx} +\frac{1}{45}v_{xxxxx} = 0,
$$
or, after re-scaling,
$$
	v_{t}+\alpha v_{x}+\beta v_{xxx}-v_{xxxxx}+vv_{x}=0.
$$
This equation is also known as the fifth-order Korteweg-de Vries (KdV) equation \cite{Boyd} or the singularly perturbed KdV equation \cite{Pomeau}. It describes a dispersive partial differential equation that encompasses various wave phenomena, such as magneto-acoustic waves in a cold plasma \cite{Kakutani}, the propagation of long waves in a shallow liquid beneath an ice sheet \cite{Iguchi}, and gravity waves on the surface of a heavy liquid \cite{Cui}, among others.

Significant efforts in recent decades have aimed to understand this model within various research frameworks. For instance, numerous studies have focused on analytical and numerical methods for solving the equation. These methods include the tanh-function method \cite{Berloff}, extended tanh-function method \cite{Biswas}, sine-cosine method \cite{Yusufoglu}, Jacobi elliptic functions method \cite{Hunter}, direct algebraic method and numerical simulations \cite{Polat}, decomposition methods \cite{Kaya}, as well as variational iteration and homotopy perturbation methods \cite{Jin}. Another important research direction is the study of the Kawahara equation from the perspective of control theory, specifically addressing the boundary controllability problem \cite{Glass e Guerrero}, which is our motivation.

So, in this context, we are interested in the boundary controllability issue of the Kawahara equation in a bounded domain. Precisely, we investigate the linear control problem
\begin{align}\label{kawahara control system}
	\left\{
	\begin{array}{lll}
		u_t+u_x+u_{xxx}-u_{xxxxx}=0,&(x,t)\text{ in }(-1,0)\times (0,T),\\
		u(0,t)=u_x(0,t)=u_{xx}(0,t)=0,&t\text{ in }(0,T),\\
		u(-1,t)=h_1(t),\ \ u_x(-1,t)=h_2(t),&t\text{ in }(0,T),\\
		u(x,0)=u_0(x),& x\text{ in }(-1,0),
	\end{array}
	\right.
\end{align}
where $h_1,h_2$ are the controls input and $u$ is the state function. This work addresses two main issues:

\vspace{0.2cm}
\noindent{\textbf{Problem $\mathcal{A}$: Null controllability.}} \textit{Given an initial data $u_0$ in a suitable space, is it possible to find control functions $h_1,h_2$ such that the state solution $u$ of the system \eqref{kawahara control system} $u(x,T)=0$?}

\vspace{0.2cm}
\noindent{\textbf{Problem $\mathcal{B}$: Reachable functions.}} \textit{Can we find a space $\mathcal{R}$ with the property that, if the final data $u_1\in \mathcal{R}$ then one can get control functions $h_1,h_2$ such that the solution $u$ of the system \eqref{kawahara control system} with $u_0=0$ satisfies $u(\cdot, T)=u_1$?}
\vspace{0.2cm}

It is important to highlight several results related to control problems associated with the system \eqref{kawahara control system}. Here are some of them. Regarding the analysis of the Kawahara equation in a bounded interval, pioneering work was done by Silva and Vasconcellos \cite{vasi1,vasi2}. They studied the stabilization of global solutions of the linear Kawahara equation in a bounded interval under the influence of a localized damping mechanism. The second contribution in this area was made by Capistrano-Filho \textit{et al.} \cite{Araruna Capistrano Doronin}, who considered a generalized Kawahara equation in a bounded domain.

Glass and Guerrero \cite{Glass e Guerrero} considered the problem
\begin{equation}\label{general kawahara}
\left\{
\begin{array}{ll}
y_t+\alpha y_{5x}=\sum_{k=0}^3a_k(x,t)\partial_x^ky+h(x,t),&(x,t)\in(0,1)\times (0,T),\\
y(0,t)=v_1(t), \ \ y(1,t)=v_2(t), \ \ y_x(0,t)=v_3(t), & t\in\times (0,T) \\
y_x(1,t)=v_4(t), \ \ y_{xx}(0,t)=v_5(t),& t\in\times (0,T),\\
y(x,0)=y_0, &x\in(0,1),
\end{array}
\right.
\end{equation}
where $\alpha>0$ and  $y_0, h, v_1,...,v_5$ are given functions. Using a Carleman estimate, they demonstrated that the system \eqref{general kawahara} is null controllable in the energy space  $L^2(0,1)$ using only the controls on the right side of the boundary, $v_2$ and $v_4$, meaning $v_1 = v_3 = v_5 = 0$. However, the authors noted that the controllability properties might fail if the set of controls is altered. For example, if $v_1 = v_2 = v_3 = v_4 = 0$, meaning only $v_5$ is used as a control input, controllability does not occur because the adjoint system associated with \eqref{general kawahara} may have unobservable solutions.

The internal controllability problem for the Kawahara equation with homogeneous boundary conditions has been addressed by Chen \cite{MoChen}. Using Carleman estimates associated with the linear operator of the Kawahara equation with internal observation, a null controllability result was demonstrated when the internal control is effective in a subdomain $\omega \subset (0,L)$. In \cite{CaGo}, the authors consider the Kawahara equation with an internal control $f(t,x)$ and homogeneous boundary conditions, the equation is shown to be exactly controllable in $L^2$-weighted Sobolev spaces. Additionally, it is shown to be controllable by regions in $L^2$-Sobolev space.

\subsection{Flatness approach and main results}  The flatness approach \cite{FlLeMaRo,ArCoGo,FiJo}, also known as differential flatness, represents a powerful concept in control theory and nonlinear system analysis. It characterizes certain dynamical systems where all states and inputs can be described as algebraic functions of a finite set of independent variables, referred to as flat outputs, and their derivatives. This approach simplifies the control and trajectory planning of complex systems. It has found extensive application across various partial differential equations, including the heat equation \cite{Heat by flatness, Reachable states for heat equation, heat}, 1-dimensional parabolic equations \cite{MaRoRo16}, the 1-dimensional Schrödinger equation \cite{1D Schrödinger by flatness}, the linear KdV equation \cite{KdV by flatness}, and more recently, the linear Zakharov-Kuznetsov equation \cite{ZK}. Here are the key aspects of the flatness approach:
\begin{enumerate}
\item \textbf{Flat outputs:} In a flat system, outputs (flat outputs) describe the system’s entire state and input trajectories using algebraic relationships involving a finite number of their derivatives.

\item \textbf{Simplified control design:}  By utilizing flat outputs, complex nonlinear control problems can be transformed into simpler linear problems. This transformation simplifies the design of control laws and facilitates the generation of desired trajectories.

\item \textbf{Trajectory planning:} The flatness approach enables systematic trajectory planning. Desired trajectories for flat outputs can be planned first, and then corresponding state and input trajectories can be computed using the flatness property.

\item \textbf{Real-world applications:} The flatness approach has been successfully applied across various fields including robotics, aerospace, process control, and automotive systems. For instance, in robotics, it aids in trajectory design for manipulators and mobile robots.
\end{enumerate}

It is crucial to recognize the advantages of the flatness approach. Typically, designing controls for nonlinear systems is complex; however, this approach mitigates this complexity. Furthermore, flatness offers a systematic approach to generate trajectories and control inputs and, as previously noted, is applicable across a broad spectrum of dynamical systems.

Nevertheless, it is important to acknowledge the limitations and challenges associated with this method. Identifying flat outputs for a specific system can be demanding and necessitates a profound understanding of the system's dynamics. Additionally, not all systems exhibit flatness, which restricts the universal applicability of this approach. In summary, the flatness approach provides a robust framework for controlling and analyzing nonlinear dynamical systems through the concept of flat outputs. It simplifies the design of control laws and trajectory planning, making it a valuable tool across various engineering applications. 

Regarding the primary contribution of this paper, we advance the study of the control problem for the fifth-order dispersive system \eqref{kawahara control system}. Unlike recent works that utilize boundary controls and use the Hilbert uniqueness method, introduced by Lions \cite{Lions}, to show control properties, this article achieves the problems $\mathcal{A}$ and $\mathcal{B}$ using two control inputs via the flatness approach. Let us now present the main results of this work. 

Motivated by the smoothing effect exhibited by the equation \eqref{kawahara control system} with free evolution ($h_1 = h_2 = 0$), we explore the reachable sets within smooth function spaces, specifically, \textit{Gevrey spaces}. First, let us introduce the definition of these spaces. 

\begin{definition}[Gevrey spaces] Given $s_1,s_2\geq 0$, a function $u:[a,b]\times [t_1,t_2]\rightarrow \mathbb{R}$ is said to be Gevrey of order $s_1$ on $[a,b]$ and $s_2$ on $[t_1,t_2]$ if $u\in C^\infty([a,b]\times[t_1,t_2])$ and there exist positive constants $C,R_1$ and $R_2$ such that
\begin{align*}
\left|\partial_x^n\partial_t^my(x,t)\right|\leq C\frac{n!^{s_1}}{R_1^n}\frac{m!^{s_2}}{R_2^m},\ \forall n,m\geq 0,\ \forall (x,t)\in[a,b]\times[t_1,t_2].
\end{align*}
The vectorial space of all functions on $[a,b]\times [t_1,t_2]$ which are Gevrey of order $s_1$ in $x$ and $s_2$ in $t$ is denoted by $G^{s_1,s_2}([a,b]\times [t_1,t_2])$.
\end{definition}

In this context, we investigate the null controllability problem associated with \eqref{kawahara control system} using the \textit{flatness approach} to establish controllability properties. The goal is to identify a set of functions in a Gevrey space that are null controllable and to demonstrate that, at some intermediate time \(\tau \in (0,T)\), the solutions to the free evolution problem fall into this set. Specifically, the first main result of this work can be stated as follows.
\begin{theorem}[Null controllability]\label{null controllability} Let $s \in [\frac{5}{2},5)$ and $T>0$ be. Given $u_0 \in L^2(-1,0)$ there exist control inputs $h_1,h_2\in G^s([0,T])$ such that the solution of \eqref{kawahara control system}  belongs to the class $u \in C\left([0,T],L^2(-1,0)\right)\cap G^{\frac{s}{5},s}([-1,0]\times [\varepsilon,T]),\ \forall\varepsilon\in (0,T),$ and satisfies $u(\cdot,T)=0$.
\end{theorem}

The previous result confirms that two flat outputs can be used to achieve null controllability, thereby addressing Problem \(\mathcal{A}\) presented at the beginning of this work. Now, to present our second main result, we need to introduce some notations. Given \(z_0 \in \mathbb{C}\) and \(R>0\), we denote by \(D(z_0,R)\) the open disk given by
\begin{align*}
D(z_0,R)=\left\{z \in \mathbb{C};\ |z-z_0|<R\right\}
\end{align*}
and $H(D(z_0,R))$ denote the set of holomorphic functions on $D(z_0,R)$. Henceforth, we consider the operators
$$Pu=\partial_xu+\partial_x^3u-\partial_x^5u,$$
where  $P^0=I_d$ and $P^n=P\circ P^{n-1},$ when $n\geq 1.$  In this way, the Kawahara equation can be expressed as
\begin{align}\label{kawahara in terms of P}
	\partial_tu+Pu=0.
\end{align}
Using  induction and the fact that $\partial_t$ and $P$ commute we see that, if $u=u(x,t)$ satisfies \eqref{kawahara in terms of P} then
\begin{align}\label{d^n_tu}
	\partial_t^nu+(-1)^{n-1}P^nu=0.
\end{align}
For every $R>1$ we define the set
\begin{align}\label{RR}
\mathcal{R}_R:=\left\{u \in C([-1,0]);
\begin{array}{c}
\exists z \in H(D(0,R));\ u=z|_{[-1,0]}\text{ and } 
(\partial_x^jP^nu)(0)=0,\ j=0,1,2
\end{array}
\right\},
\end{align}
now on called \textit{a set of reachable states}. The following result is the second main result in this paper, answering the Problem $\mathcal{B}$.
\begin{theorem}\label{main2}
Let $T>0$, $R_0:=2\cdot 6^{5^{-1}}\cdot e^{(5e)^{-1}}>1$ and $R>2R_0$ be. Given $u_1 \in \mathcal{R}_R$ there exist control inputs $h_1,h_2 \in G^5([0,T])$ for which the solution $u$ of \eqref{kawahara control system} with $u_0=0$ satisfies $u(\cdot,T)=u_1$ and $u \in G^{1,5}([-1,0]\times [0,T])$.
\end{theorem}

\subsection{Further comments and outline}
Observe that Theorem \ref{main2} ensures that for the linear Kawahara equation, any reachable state can likely be extended as a holomorphic function on some open set in $\mathbb{C}$. Moreover,  the reachable states corresponding to controllability to the trajectories are in \(G^{\frac{1}{2}}([-1,0])\), allowing them to be extended as functions in \(H(\mathbb{C})\). In contrast, the reachable functions in Theorem \ref{main2} do not need to be holomorphic over the entire set \(\mathbb{C}\); they can have poles outside \(D(0, R)\). 

\begin{remark} Let us present some comments about our work.
\begin{itemize}
\item[i.] Our result is entirely linear and applies only to the linear system \eqref{kawahara control system}. Therefore, a natural extension is to consider the nonlinear problem, which includes the term \(uu_x\). However, it is required to modify the method used here. We believe that the strategy used in \cite{LaRo} could be adapted to our work, though this remains an open problem.
\item[ii.] Note that the set defined by \eqref{RR} is an example of reachable functions, though they are not completely understood. We believe there are other sets for which Theorem \ref{main2} remains valid. For instance, in \cite{ZK}, the authors presented another example for an extension of the KdV equation in a two-dimensional case. Thus, it remains an open question to verify other sets where Theorem \ref{main2} holds.
\item[iii.] As previously mentioned, several authors have applied the strategy of this article to different systems. In our case, additional difficulties arise when dealing with a fifth-order operator in space, specifically \( P u = \partial_x u + \partial_x^3 u - \partial_x^5 u \). Consequently, addressing the smoothing properties is challenging because we are working with five sets of different regularities and additional terms appear. Furthermore, the set \eqref{RR}, and consequently, Theorem \ref{main2}, can be obtained using the strategies outlined in \cite{KdV by flatness}. However, the Gevrey space level needs to be adjusted, and the operator's order must be carefully adapted to our specific case.
\item[iv.] Finally, observe that our results are verified to the Benney–Lin type equation:
$$\partial_t u + \partial_x^3 u + \mu_0 \partial_x^4 u - \partial_x^5 u = 0, \quad (x, t) \in (0,1) \times (0,T),$$
with the same boundary conditions as in equation \eqref{kawahara control system} and with \(\mu_0 > 0\). This equation describes the evolution of one-dimensional small but finite amplitude long waves in various physical systems in fluid dynamics (see \cite{Ben66} and \cite{Lin74}). The coefficient \(\mu_0 > 0\) introduces nonconservative dissipative effects to the dispersive Kawahara equation \eqref{kawahara control system} (where \(\mu_0 = 0\)), and thus it is sometimes referred to as the strongly dissipative Kawahara equation \cite{Zho19}, the fifth-order Korteweg-de Vries equation \cite{ZZF18}, or the generalized Kawahara equation \cite{CDLW22}. For more details about the Benney-Lin type equation, we encourage the reader to see the reference \cite{CoRu}.
\end{itemize}
\end{remark}

 Let us conclude our introduction with an outline of our work. Section \ref{sec2} addresses Question \(\mathcal{A}\) by presenting the control result, which is a consequence of the flatness property and the smoothing effect of the Kawahara equation. In Section \ref{sec3}, we provide an example of a set that can be reached from \(0\) by the system \eqref{kawahara control system}. This result, stemming from the flatness property extended to the limit case \(s = 5\), partially answers Question \(\mathcal{B}\).

\section{Controllability result}\label{sec2}
We want to prove the null controllability property for the system \eqref{kawahara control system}. Our initial objective is to show the flatness property. This property ensures that we can parameterize the solution of the system \eqref{kawahara control system} using the ``flat outputs" $\partial_x^3u(0,t)$ and $\partial_x^4u(0,t)$. To prove it, we will examine the ill-posed system 
\begin{align}\label{system 3}
	\left\{
	\begin{array}{lll}
		u_t+\partial_xu+\partial_x^3u-\partial_x^5u=0,&(x,t)\text{ in }(-1,0)\times (0,T),\\
		u(0,t)=\partial_xu(0,t)=\partial_x^2u(0,t)=0,&t\text{ in }(0,T),\\
			\partial_x^3u(0,t)=y(t),\ \ \partial_x^4u(0,t)=z(t),&t\text{ in }(0,T).
	\end{array}
	\right.
\end{align}
Precisely, we will show that the solution of \eqref{system 3} belongs to $G^{\frac{s}{5},s}([-1,0]\times [\varepsilon,T])$ for all $\varepsilon \in (0,T)$ and can be written in the form
\begin{align}\label{eq 4}
	u(x,t)=\sum_{j=0}^\infty f_j(x)y^{(j)}(t)+\sum_{j=0}^\infty g_j(x)z^{(j)}(t),\ (x,t)\in [-1,0]\times [\varepsilon,T]
\end{align}
where $y,z\in G^{s}([0,T])$ for some $s\geq 0$, with $y^{(j)}(T)=z^{(j)}(T)=0$ for every $j\geq 0$. Here, $f_j$ and $g_j$ are called the generating functions and are constructed following the ideas introduced in \cite{MaRoRo16}. The smoothing effect is responsible for ensuring that, from time $\varepsilon$ onward, the solution of the free evolution problem associated to \eqref{kawahara control system} is Gevrey of order $\frac{1}{2}$ in $x$ and $\frac{5}{2}$ in $t$. Finally, a result of unique continuation is used to ensure that this solution coincides with the one associated with \eqref{system 3}, described in \eqref{eq 4}.

\subsection{Flatness property} We need to prove that there exists a one-to-one correspondence between solutions of \eqref{system 3} and a certain space of smooth functions, which we will call the flatness property. We will often use Stirling's formula:
$$
	n!\sim \left(\frac{n}{e}\right)^n\sqrt{2\pi n},\ \forall n,
$$
and the inequality
\begin{align}\label{inequality}
	(n+m)!\leq 2^{n+m}n!m!.
\end{align}
For $n \in \mathbb{N}$, $p \in [1,\infty]$ and $f \in W^{n,p}(-1,0)$ we denote
\begin{align*}
	\|f\|_p=\|f\|_{L^p(-1,0)}\text{ \ \ and \ \ }\|f\|_{n,p}=\sum_{i=0}^n\|\partial_x^if\|_p.
\end{align*}

Let us in this part find a solution $u$ for \eqref{system 3} in the form
\begin{align}\label{flatness solution}
u(x,t)=\sum_{j\geq 0}f_j(x)y^{(j)}(t)+\sum_{j\geq 0}g_j(x)z^{(j)}(t).
\end{align}
Here, $y,z,f_j,g_j$ satisfies the following conditions:
\begin{enumerate}
	\item[$1)$] $y,z \in G^s([0,T])$ with $s \in (1,5)$;
	\item[$2)$] $f_j,g_j \in L^\infty([-1,0])$ with polynomial growths in the form
	$$
	\left|f_j(x)\right|\leq \frac{|x|^{5j+r}}{(5j+r)!}\text{ \ \ and \ \ }	\left|g_j(x)\right|\leq \frac{|x|^{5j+r}}{(5j+r)!},\ \forall j\geq 0,\ \forall x \in [-1,0],
	$$
	for some $r\in \{0,1,2,3,4\}$;
	\item[$3)$] $\partial_x^3(0,t)=y(t)$ and $\partial_x^4u(0,t)=z(t)$.
\end{enumerate}
Considering $u$ as in \eqref{flatness solution} and assuming that we can derive term by term, we get that 
\begin{align*}
u_t+\partial_xu+\partial_x^3u-\partial_x^5u=&\sum_{j\geq 0}f_j(x)y^{(j+1)}(t)+\sum_{j\geq 0}\left(f_{jx}+f_{j3x}-f_{j5x}\right)(x)y^{(j)}(t)\\
&+\sum_{j\geq 0}g_j(x)z^{(j+1)}(t)+\sum_{j\geq 0}\left(g_{jx}+g_{j3x}-g_{j5x}\right)(x)z^{(j)}(t).
\end{align*}
Considering a new index $l=j+1$ gives us
\begin{align*}
\sum_{j\geq 0}f_j(x)y^{(j+1)}(t)=\sum_{l\geq 1}f_{l-1}(x)y^{(l)}(t)\text{ \ \ and \ \ }\sum_{j\geq 0}g_j(x)z^{(j+1)}(t)=\sum_{l\geq 1}g_{l-1}(x)z^{(l)}(t).
\end{align*}
Thus
\begin{align*}
u_t+\partial_xu+\partial_x^3u-\partial_x^5u=&\sum_{j\geq 1}f_{j-1}(x)y^{(j)}(t)+\sum_{j\geq 1}\left(f_{jx}+f_{j3x}-f_{j5x}\right)(x)y^{(j)}(t)\\
&+\left(f_{0x}+f_{03x}-f_{05x}\right)(x)y(t)+\left(g_{0x}+g_{03x}-g_{05x}\right)(x)z(t)\\
&+\sum_{j\geq 1}g_{j-1}(x)z^{(j)}(t)+\sum_{j\geq 1}\left(g_{jx}+g_{j3x}-g_{j5x}\right)(x)z^{(j)}(t).
\end{align*}
Hence, if we have
\begin{align*}
\begin{cases}
&f_{0x}+f_{03x}-f_{05x}=g_{0x}+g_{03x}-g_{05x}=0,\\
&f_{jx}+f_{j3x}-f_{j5x}=-f_{j-1},\quad \ \forall j\geq 1,\\
&g_{jx}+g_{j3x}-g_{j5x}=-g_{j-1},\quad \ \forall\ j\geq 1,
\end{cases}
\end{align*}
then
$$
u_t+\partial_xu+\partial_x^3u-\partial_x^5u=0.
$$
Furthermore, imposing the conditions
\begin{align*}
\begin{cases}
\begin{array}{l}
f_0(0)=f_{0x}(0)=f_{02x}(0)=f_{04x}(0)=0,\quad f_{03x}(0)=1,\\
f_j(0)=f_{jx}(0)=f_{j2x}(0)=f_{j3x}(0)=f_{j4x}(0)=0,\quad \forall j\geq 1
\end{array}
\end{cases}
\end{align*}
and
\begin{align*}
\begin{cases}
	\begin{array}{l}
		g_0(0)=g_{0x}(0)=g_{02x}(0)=g_{03x}(0)=0,\quad g_{04x}(0)=1,\\
		g_j(0)=g_{jx}(0)=g_{j2x}(0)=g_{j3x}(0)=g_{j4x}(0)=0,\quad \forall j\geq 1,
	\end{array}
	\end{cases}
\end{align*}
we obtain,
\begin{align*}
\begin{cases}
&u(0,t)=\partial_xu(0,t)=\partial_x^2u(0,t)=0,\\
&\partial_x^3u(0,t)=y(t),\quad \partial_x^4u(0,t)=z(t),
	\end{cases}
\end{align*}
 for $t \in (0,T)$. This leads us to define inductively the functions $\{f_j\}_{j\geq 0}$ and $\{g_j\}_{j\geq 0}$ as follows:
\begin{enumerate}
	\item[$1)$] $f_0$ is the solution of the IBVP
	\begin{align}\label{f_0}
	\left\{
	\begin{array}{l}
	f_{0x}+f_{03x}-f_{05x}=0,\\
	f_0(0)=f_{0x}(0)=f_{02x}(0)=f_{04x}(0)=0,\quad 	f_{03x}(0)=1.
	\end{array}
	\right.
	\end{align}
and, for $j\geq 1$, $f_j$ is the solution for the IBVP
\begin{align}\label{fj}
\left\{
\begin{array}{l}
f_{jx}+f_{j3x}-f_{j5x}=-f_{j-1},\\
f_j(0)=f_{jx}(0)=f_{j2x}(0)=f_{j3x}(0)=f_{j4x}(0)=0.
\end{array}
\right.
\end{align}
\item[$2)$] $g_0$ is the solution of the IBVP
\begin{align}\label{g_0}
	\left\{
	\begin{array}{l}
		g_{0x}+g_{03x}-g_{05x}=0,\\
		g_0(0)=g_{0x}(0)=g_{02x}(0)=g_{03x}(0)=0,\quad  g_{04x}(0)=1.
	\end{array}
	\right.
\end{align}
and, for $j\geq 1$, $g_j$ is the solution for the IVP
\begin{align}\label{gj}
	\left\{
	\begin{array}{l}
		g_{jx}+g_{j3x}-g_{j5x}=-g_{j-1},\\
		g_j(0)=g_{jx}(0)=g_{j2x}(0)=g_{j3x}(0)=g_{j4x}(0)=0.
	\end{array}
	\right.
\end{align}
\end{enumerate}

We will prove that, given $s \in (1,5)$, there exists a one to one correspondence between solutions of \eqref{system 3} and pairs of functions $(y,z)\in G^s([0,T])\times G^s([0,T])$, namely,
\begin{align*}
u\mapsto\left(\partial_x^3y(0,\cdot),\partial_x^4z(0,\cdot)\right).
\end{align*} 
In the sense of this bijection, we shall say that the system \eqref{system 3} is flat. Observe that, an expression in terms of the families $\{f_j\}_{j\geq 0}$ and $\{g_j\}_{j\geq 0}$ for a solution $u$ of \eqref{system 3} as in \eqref{flatness solution} must be unique, that is, if
\begin{align*}
	u(x,t)=\sum_{j\geq 0}f_j(x)y^{(j)}(t)+\sum_{j\geq 0}g_j(x)z^{(j)}(t)
\end{align*}
and
\begin{align*}
	u(x,t)=\sum_{j\geq 0}f_j(x)\tilde{y}^{(j)}(t)+\sum_{j\geq 0}g_j(x)\tilde{z}^{(j)}(t),
\end{align*}
with $y,\tilde{y}, z,\tilde{z}\in G^s([0,T])$, then $y=\tilde{y}$ and $z=\tilde{z}$. 
%Indeed
%\begin{align*}
%\sum_{j\geq 0}f_j(x)y^{(j)}(t)+\sum_{j\geq 0}g_j(x)z^{(j)}(t)=\sum_{j\geq 0}f_j(x)\tilde{y}^{(j)}(t)+\sum_{j\geq 0}g_j(x)\tilde{z}^{(j)}(t)
%\end{align*}
%implies that
%\begin{align*}
%	\sum_{j\geq 0}f_{j3x}(x)y^{(j)}(t)+\sum_{j\geq 0}g_{j3x}(x)z^{(j)}(t)=\sum_{j\geq 0}f_{j3x}(x)\tilde{y}^{(j)}(t)+\sum_{j\geq 0}g_{j3x}(x)\tilde{z}^{(j)}(t)
%\end{align*}
%and, in particular
%\begin{align*}
%	\sum_{j\geq 0}f_{j3x}(0)y^{(j)}(t)+\sum_{j\geq 0}g_{j3x}(0)z^{(j)}(t)=\sum_{j\geq 0}f_{j3x}(0)\tilde{y}^{(j)}(t)+\sum_{j\geq 0}g_{j3x}(0)\tilde{z}^{(j)}(t).
%\end{align*}
%Hence \eqref{f_0}-\eqref{g_j} gives us $y=\tilde{y}$. In a similar way, it is seen that $z=\tilde{z}$. Therefore, 

\begin{remark}
\begin{enumerate}
	\item[$1)$] Note that considering a toy model, that is, 
	\begin{align*}
\left\{
\begin{array}{ll}
u_t-\partial_x^5u=0,&(x,t)\in (-1,0)\times (0,T),\\
u(0,t)=\partial_xu(0,t)=\partial_x^2(0,t)=0,&t\in (0,T),\\
u(-1,t)=h_1(t)\ \ \ \ \partial_x u(-1,t)=h_2(t),&t\in (0,T),\\
u(x,0)=u_0(x),
\end{array}
\right.
\end{align*}
the first equation in the systems \eqref{f_0}-\eqref{g_0} do not have the first and third derivatives terms. Then, direct computation gives
	\begin{align*}
		f_j(x)=\frac{x^{5j+3}}{(5j+3)!}\text{ \ \ and \ \ }g_j(x)=\frac{x^{5j+4}}{(5j+4)!},\ \ \forall\ j\geq 0,\ x \in [-1,0].
	\end{align*}
	\item[$2)$] Returning to the full system, with terms of the first and third derivatives, we have that the solutions $f_0$ and $g_0$ of \eqref{f_0} and \eqref{g_0} are given by
	\begin{align}\label{expression for f_0}
	f_0(x)=\frac{1}{\sqrt{a}(a+b)}\sinh\left(\sqrt{a}x\right)-\frac{1}{\sqrt{b}(a+b)}\sin\left(\sqrt{b}x\right)
	\end{align}
and
\begin{align}\label{expression for g_0}
g_0(x)=\frac{1}{a(a+b)}\cosh\left(\sqrt{a}x\right)+\frac{1}{b(a+b)}\cos\left(\sqrt{b}x\right)-\frac{1}{a(a+b)}-\frac{1}{b(a+b)},
\end{align}
respectively, where $a=\frac{\sqrt{5}+1}{2}$ and $b=\frac{\sqrt{5}-1}{2}$.
\end{enumerate}
\end{remark}
To conclude that the system \eqref{system 3} is flat, it is enough to show that the solutions of \eqref{system 3} can be expressed as in \eqref{flatness solution} with $\{f_j\}_{j\geq 0}$ and $\{g_j\}_{j\geq 0}$ given by \eqref{f_0}-\eqref{gj}. Precisely, we will see that given $y,z \in G^s([0,T])$, then $u$ given by \eqref{flatness solution} is well defined, belongs to $G^{\frac{s}{5},s}([-1,0]\times[0,T])$ and it solves the problem
\begin{align}\label{flatness system}
	\left\{
\begin{array}{lll}
	u_t+\partial_xu+\partial_x^3u-\partial_x^5u=0&(x,t)\text{ in }(-1,0)\times (0,T)\\
	u(0,t)=\partial_xu(0,t)=\partial_x^2u(0,t)=0,&t\text{ in }(0,T)\\
	\partial_x^3u(0,t)=y(t),\ \ \partial_x^4u(0,t)=z(t)&t\in(0,T).
\end{array}
\right.
\end{align}
To do this, first, we need to establish estimates to the norms $\|f_j\|_{L^\infty(-1,0)}$ and $\|g_j\|_{L^\infty(-1,0)}$, as suggested before. At this point, it is very useful to note that, for $j\geq 1$, $f_j$ (respectively $g_j$) can be written in terms of $f_0$ and $f_{j-1}$ (respectively $g_0$ and $g_{j-1}$).
\begin{lemma}\label{f_j and g_j} For any $j\geq 1$ and $x \in [-1,0]$ we have
\begin{align}\label{expression for f_j}
f_j(x)=\int_0^x\int_0^yf_0(y-\xi)f_{j-1}(\xi)d\xi dy
\end{align}
and
\begin{align}\label{expression for g_j}
	g_j(x)=\int_0^xg_0(x-\xi)g_{j-1}(\xi)d\xi.
\end{align}
\end{lemma}
\begin{proof}
Let $x \in [-1,0]$, $y\in [x,0]$ and $\xi \in [y,0]$ be. From \eqref{fj} we have
\begin{align*}
f_{j\xi}(\xi)f_0(y-\xi)+f_{j3\xi}(\xi)f_0(y-\xi)-f_{j5\xi}(\xi)f_0(y-\xi)=-f_0(y-\xi)f_{j-1}(\xi).
\end{align*}
Integrating with respect to $\xi$ we obtain
\begin{align*}
\int_0^yf_j(\xi)\left(f_{0\xi}(y-\xi)+f_{03\xi}(y-\xi)-f_{05\xi}(y-\xi)\right)d\xi-f_{j\xi}(y)=-\int_0^yf_0(y-\xi)f_{j-1}(\xi)d\xi
\end{align*}
and by \eqref{f_0}, after some integration by parts, it follows that
%\begin{align*}
%-f_{j\xi}(y)=-\int_0^yf_0(y-\xi)f_{j-1}(\xi)d\xi.
%\end{align*}
%Therefore,
\begin{align*}
f_j'(y)=\int_0^yf_0(y-\xi)f_{j-1}(\xi)d\xi.
\end{align*}
Integrating from $0$ to $x$ with respect to $y$ we get
%\begin{align*}
%f_j(x)-f_j(0)=\int_0^x\int_0^yf_0(y-\xi)f_{j-1}(\xi)d\xi dy
%\end{align*}
%so that
\begin{align*}
f_j(x)=\int_0^x\int_0^yf_0(y-\xi)f_{j-1}(\xi)d\xi dy,\ \forall x \in [-1,0],
\end{align*}
showing \eqref{expression for f_j}. Similarly \eqref{expression for g_j} is verified.
\end{proof}

The next lemma will ensure that the series in \eqref{flatness solution} are convergents.
\begin{lemma}\label{estimetes to f_j and g_j}
For every $j\geq 0$ we have
\begin{equation}\label{f_j}
\left|f_j(x)\right|\leq 2^j\frac{|x|^{5j+1}}{(5j+1)!}
\end{equation}
and
\begin{equation}\label{g_j}
\left|g_j(x)\right|\leq 2^j\frac{|x|^{5j+1}}{(5j+1)!},
\end{equation}
for all $j\geq 0$ and  $x \in [-1,0]$.
\end{lemma}
\begin{proof}Let us start to proving \eqref{f_j} by induction on $j$. First, we must verify the case $j=0$:
\begin{align*}
\left|f_0(x)\right|\leq |x|, \quad \forall x \in [-1,0].
\end{align*}
Note that $f_0(x)\leq 0$ in $[-1,0]$. Indeed, from \eqref{expression for f_0}, since $\cosh>1$ in $\mathbb{R}\backslash\{0\}$ and $\cos\geq 1$ in $\mathbb{R}$,
\begin{align*}
f_{0x}=\frac{1}{a+b}\cosh\left(\sqrt{a}x\right)-\frac{1}{a+b}\cos\left(\sqrt{b}x\right)>\frac{1}{a+b}-\frac{1}{a+b}=0,
\end{align*}
for all $x \in (-\infty,0)$. This implies that $f_0$ is increasing, thus $f_0(x)\leq f_0(0)=0$.  Therefore, it is sufficient to show that
\begin{align}\label{eq 27a}
f_0(x)\geq x,\quad \ \forall x \in [-1,0].
\end{align}
Defining $\varphi(x)=f_0(x)-x$ we have $\varphi'(x)=f_0'(x)-1$. On the other hand, we have
\begin{align*}
f_{03x}(x)=\frac{a}{a+b}\cosh\left(\sqrt{a}x\right)+\frac{b}{a+b}\cos\left(\sqrt{b}x\right)>\frac{a-b}{a+b}=\frac{1}{\sqrt{5}}>0,
\end{align*}
 for $x \in (-\infty,0)$, and so $f_{02x}$ is increasing in $(-\infty,0]$, implying that $f_{02x}(x)<f_{02x}(0)=0$ for $x\in (- \infty,0]$. Consequently $f_{0x}$ is decreasing in $(-\infty,0]$ which implies that
\begin{align*}
f_{0x}(x)<f_{0x}(-1)\approxeq 0,54<1, \quad \ \forall\ x \in (-1,0),
\end{align*}
so $\varphi'<0$ in $(-1,0)$. Hence $\varphi$ is decreasing on $[-1,0]$ and therefore $\varphi(x)\geq \varphi(0)=0,$ for $x \in [-1,0]$, which implies \eqref{eq 27a}.

For the next step, we need to verify the estimate
\begin{align}\label{eq 28a}
\left|f_{03x}(x)\right|\leq 2,
\end{align}
when $x \in  [-1,0].$ Note that we have
\begin{align*}
f_{05x}(x)=\frac{a^2}{a+b}\cosh\left(\sqrt{a}x\right)-\frac{b^2}{a+b}\cos\left(\sqrt{b}x\right)>\frac{a^2-b^2}{a+b}=a-b=1,
\end{align*}
for $x \in (-\infty,0)$. This gives us that $f_{04x}$ is increasing in $(-\infty,0]$, in particular, $f_{04x}(x)<f_{04x}(0)=0.$ Hence $f_{03x}$ is decreasing in $(-\infty,0]$, and so
$
0=f_{03x}(0)\leq f_{03x}(x)\leq f_{03x}(-1)\approxeq 1,59<2,
$
for $x \in [-1,0]$, getting \eqref{eq 28a}. 

Now, suppose that for some $j\geq 1$ 
\begin{align*}
\left|f_j(x)\right|\leq 2^{j-1}\frac{|x|^{5(j-1)+1}}{[5(j-1)+1]!},\ \forall x \in [-1,0],
\end{align*}
holds. By Lemma \ref{f_j and g_j} we have
\begin{align*}
f_j(x)=\int_0^x\int_0^yf_0(y-\xi)f_{j-1}(\xi)d\xi dy, \quad \forall x \in [-1,0].
\end{align*}
Using integration by parts (with respect to $\xi$) together with the boundary conditions in \eqref{f_0} we obtain
\begin{align*}
\int_0^yf_0(y-\xi)f_{j-1}(\xi)d\xi&=\int_0^yf_0(y-\xi)\frac{d}{d\xi}\int_0^\xi f_{j-1}(\sigma)d\sigma d\xi\\
&=\left[f_0(y-\xi)\int_0^\xi f_{j-1}(\sigma)d\sigma\right]_0^y+\int_0^y\int_0^\xi f_{j-1}(\sigma)d\sigma f_{0\xi}(y-\xi)d\xi\\
&=\int_0^yf_{0\xi}(y-\xi)\int_0^\xi f_{j-1}(\sigma)d\sigma d\xi.
\end{align*}
Define $F(\xi)=\int_0^\xi f_{j-1}(\sigma)d\sigma,$ with this, we can write
\begin{align*}
\int_0^yf_0(y-\xi)f_{j-1}(\xi)d\xi&=\int_0^yf_{0\xi}(y-\xi)F(\xi)d\xi=\int_0^yf_{0\xi}(y-\xi)\frac{d}{d\xi}\int_0^\xi F(\tau)d\tau d\xi.
\end{align*}
Integration by parts together with the boundary conditions giving in \eqref{f_0} we get
\begin{align*}
\int_0^yf_0(y-\xi)f_{j-1}(\xi)d\xi&=\left[f_{0\xi}(y-\xi)\int_0^\xi F(\tau)d\tau\right]_0^y+\int_0^y\int_0^\xi F(\tau)d\tau f_{02\xi}(y-\xi)d\xi\\
&=\int_0^y f_{02\xi}(y-\xi)\int_0^\xi F(\tau)d\tau d\xi.
\end{align*}
Setting also $G(\xi)=\int_0^\xi F(\tau)d\tau,$ and we have
\begin{align*}
\int_0^yf_0(y-\xi)f_{j-1}(\xi)d\xi&=\int_0^y f_{02\xi}(y-\xi)G(\xi)d\xi=\int_0^y f_{02\xi}(y-\xi)\frac{d}{d\xi}\int_0^\xi G(\rho)d\rho d\xi.
\end{align*}
Using again integration by parts together with \eqref{f_0}, it follows that
\begin{align*}
\int_0^yf_0(y-\xi)f_{j-1}(\xi)d\xi&=\left[f_{02\xi}(y-\xi)\int_0^\xi G(\rho)d\rho\right]_0^y+\int_0^y\int_0^\xi G(\rho)d\rho f_{03\xi}(y-\xi)d\xi\\
&=\int_0^y f_{03\xi}(y-\xi)\int_0^\xi G(\rho)d\rho d\xi.
\end{align*}
Therefore, the previous equality ensures that
\begin{align*}
\int_0^yf_0(y-\xi)f_{j-1}(\xi)d\xi&=\int_0^y f_{03\xi}(y-\xi)\int_0^\xi G(\rho)d\rho d\xi\\
&=\int_0^y f_{03\xi}(y-\xi)\int_0^\xi \int_0^\rho F(\tau)d\tau d\rho d\xi\\
&=\int_0^y f_{03\xi}(y-\xi)\int_0^\xi \int_0^\rho \int_0^\tau f_{j-1}(\sigma)d\sigma d\tau d\rho d\xi\\
&=\int_0^y \int_0^\xi \int_0^\rho \int_0^\tau f_{03\xi}(y-\xi)f_{j-1}(\sigma)d\sigma d\tau d\rho d\xi,
\end{align*}
so
\begin{align*}
f_j(x)=\int_0^x\int_0^y \int_0^\xi \int_0^\rho \int_0^\tau f_{03\xi}(y-\xi)f_{j-1}(\sigma)d\sigma d\tau d\rho d\xi dy,\ \ \forall x \in [-1,0].
\end{align*}
Then, by the induction hypothesis, we get that
\begin{align*}
\left|f_j(x)\right|
%&\leq \int_0^x\int_0^y \int_0^\xi \int_0^\rho \int_0^\tau \left|f_{03\xi}(y-\xi)\right|\left|f_{j-1}(\sigma)\right|d\sigma d\tau d\rho d\xi dy\\
&\leq \int_0^x\int_0^y \int_0^\xi \int_0^\rho \int_0^\tau 2\cdot2^{j-1}\frac{|\sigma|^{5(j-1)+1}}{[5(j-1)+1]!}d\sigma d\tau d\rho d\xi dy\\
&\leq 2^j\left|\int_0^x\int_0^y \int_0^\xi \int_0^\rho \int_0^\tau \frac{|\sigma|^{5j-4}}{(5j-4)!}d\sigma d\tau d\rho d\xi dy\right|\\
&=2^j\left|\int_0^x\int_0^y \int_0^\xi \int_0^\rho \int_0^\tau \frac{(-\sigma)^{5j-4}}{(5j-4)!}d\sigma d\tau d\rho d\xi dy\right|,
\end{align*}
for $x \in [-1,0]$. Pick $r=-\sigma$, thus $dr=-d\sigma$, that is,  
\begin{align*}
\int_0^\tau\frac{(-\sigma)^{5j-4}}{(5j-4)!}d\sigma=-\int_0^{-\tau}\frac{r^{5j-4}}{(5j-4)!}dr=\left[-\frac{r^{5j-3}}{(5j-3)(5j-4)!}\right]_0^{-\tau}=-\frac{(-\tau)^{5j-3}}{(5j-3)!}
\end{align*}
thus
\begin{align*}
	\left|f_j(x)\right|&\leq 2^j\left|\int_0^x\int_0^y \int_0^\xi \int_0^\rho \frac{(-\tau)^{5j-3}}{(5j-3)!}d\tau d\rho d\xi dy\right|.
\end{align*}
Now set $r=-\tau$ to get
\begin{align*}
\int_0^\rho\frac{(-\tau)^{5j-3}}{(5j-3)!}d\tau=-\int_0^{-\rho}\frac{r^{5j-3}}{(5j-3)!}dr=\left[-\frac{r^{5j-2}}{(5j-2)(5j-3)!}\right]_0^{-\rho}=-\frac{(-\rho)^{5j-2}}{(5j-2)!},
\end{align*}
and, consequently, 
\begin{align*}
\left|f_j(x)\right|&\leq 2^j\left|\int_0^x\int_0^y \int_0^\xi \frac{(-\rho)^{5j-2}}{(5j-2)!} d\rho d\xi dy\right|.
\end{align*}
Proceeding with three more integrations, we obtain
\begin{align*}
\left|f_j(x)\right|&\leq 2^j\left|\int_0^x\int_0^y  \frac{(-\xi)^{5j-1}}{(5j-1)!}  d\xi dy\right|\leq 2^j\left|\int_0^x  \frac{(-y)^{5j}}{(5j)!} dy\right|\leq 2^j\left|\frac{(-x)^{5j+1}}{(5j+1)!}\right|,
\end{align*}
for all $x \in [-1,0]$, showing \eqref{f_j}.

Now, for \eqref{g_j}, arguing similarly, we start by checking that $\left|g_0(x)\right|\leq |x|,$ for all $x \in [-1,0]$.  Since $\cosh\geq 1$ and $\cos\geq 1$, from \eqref{expression for g_0} we have $g_0\geq 0$. Then it is enough to show that $g_0(x)\leq -x$, for $x \in [-1,0]$. Defining $\psi(x)=g_0(x)+x$  we have $\psi'(x)=g_0'(x)+1$. On the other hand,
\begin{align*}
g_{02x}(x)=\frac{1}{a+b}\cosh\left(\sqrt{a}x\right)-\frac{1}{a+b}\cos\left(\sqrt{b}x\right)>0, \quad \forall x \in (-\infty,0),
\end{align*}
so $g_{0x}$ is increasing in $(-\infty,0]$. Thus $g_{0x}(x)>g_{0x}(-1)\approxeq -0,18>-1$, for all $x \in (-1,0)$, which implies that $\psi'>0$ in $(-1,0)$ and therefore $\psi$ is increasing in $[-1,0]$. Consequently
\begin{align*}
\psi(x)\leq \psi(0)=0 \iff g_0(x)\leq -x,\quad \ \forall x \in [-1,0].
\end{align*}
In the next part, we need the estimate  $\left|g_{04x}(x)\right|<2$, for all $x \in [-1,0]$. This immediately follows from \eqref{eq 28a} and the fact that $g_{04x}\equiv f_{03x}$. 

Now, assume that 
\begin{align*}
\left|g_{j-1}(x)\right|\leq 2^{j-1}\frac{|x|^{5(j-1)+1}}{[5(j-1)+1]!}\ \ \forall x \in [-1,0],
\end{align*}
holds for some $j\geq 1$. By Lemma \ref{f_j and g_j} we have
\begin{align*}
g_j(x)=\int_0^xg_0(x-\xi)g_{j-1}(\xi)d\xi.
\end{align*}
Proceeding with integration by parts as in the case of $f_j$ and using the boundary conditions in \eqref{gj} we obtain
\begin{align*}
g_j(x)=\int_0^x\int_0^\xi\int_0^\lambda\int_0^\rho\int_0^\tau g_{04x}(x-\xi)g_{j-1}(\sigma)d\sigma d\tau d\rho d\lambda d\xi.
\end{align*}
Then, by induction hypothesis we have
\begin{align*}
\left|g_j(x)\right|
%&\leq \int_0^x\int_0^\xi\int_0^\lambda\int_0^\rho\int_0^\tau \left|g_{04x}(x-\xi)\right|\left|g_{j-1}(\sigma)\right|d\sigma d\tau d\rho d\lambda d\xi\\
&\leq \int_0^x\int_0^\xi\int_0^\lambda\int_0^\rho\int_0^\tau 2\cdot2^{j-1}\frac{|\sigma|^{5(j-1)+1}}{[5(j-1)+1]!} d\sigma d\tau d\rho d\lambda d\xi\\
&\leq 2^j\left|\int_0^x\int_0^\xi\int_0^\lambda\int_0^\rho\int_0^\tau \frac{|\sigma|^{5j-4}}{(5j-4)!} d\sigma d\tau d\rho d\lambda d\xi \right|\\
&= 2^j\left|\int_0^x\int_0^\xi\int_0^\lambda\int_0^\rho\int_0^\tau \frac{(-\sigma)^{5j-4}}{(5j-4)!} d\sigma d\tau d\rho d\lambda d\xi \right|,
\end{align*}
 for $x \in [-1,0]$. Using the same process of repeated integration by substitution as in the case of $f_j$ we obtain \eqref{g_j}, which concludes the proof.
\end{proof}

We are now in a position to solve the system \eqref{flatness system}.
\begin{proposition}\label{flatness well-posedness}
Let $s \in (1,5)$, $T>0$ and $y,z \in G^s([0,T])$. The the function $u(x,t)$ defined in \eqref{flatness solution} belongs to $G^{\frac{s}{5},s}([-1,0]\times [0,T])$ and it solves \eqref{flatness system}. In particular, the corresponding controls $h_1(t):=u(-1,t)$ and $h_2(t):=\partial_x u(-1,t)$ are Gevrey of order $s$ on $[0,T]$.
\end{proposition}
\begin{proof}
Let us briefly explain our strategy to prove this result. Consider $m,n\geq 0$ and $(x,t)\in [-1,0]\times [0,T]$ be given. Our task is first to derivate the series in \eqref{flatness solution} as follows
\begin{align}\label{eq 31a}
\partial_x^n\partial_t^m u(x,t)=\sum_{j\geq 0}\partial_x^n\partial_t^m\left(f_j(x)y^{(j)}(t)\right)+\sum_{j\geq 0}\partial_x^n\partial_t^m\left(g_j(x)z^{(j)}(t)\right).
\end{align}
With this, since $m,n$ are arbitrary, we will prove that the series in \eqref{eq 31a} is uniformly convergent in $[-1,0]\times [0,T]$, which will ensure that $u \in C^\infty([-1,0]\times [0,T])$. Next, to conclude that $u \in G^{\frac{s}{5},s}([-1,0]\times [0,T])$ we must prove that
\begin{align*}
\left|\partial_x^n\partial_t^mu(x,t)\right|\leq C\frac{n!^\frac{s}{5}}{R_1^n}\frac{m!^s}{R_2^m}
\end{align*}
for some constants $C,R_1,R_2>0$.

Let us star, choose $s \in (1,5)$ and $y,z \in G^s([0,T])$. Then there exists $M,R>0$ such that
\begin{align}\label{eq 30a}
	\left|y^{(j)}(t)\right|\leq M\frac{j!^s}{R^j}\text{ \ \ and \ \ }\left|z^{(j)}(t)\right|\leq M\frac{j!^s}{R^j}\ \ \forall j\geq 0,\ \forall t \in [0,T].
\end{align}
For $k\geq 0$, since $\partial_t$ and $P$ commute, we have that
\begin{align*}
\partial_t^mP^k\left(f_j(x)y^{(j)}(t)\right)=P^k\left(f_j(x)\right)y^{(j+m)}(t).
\end{align*}
From \eqref{f_0} and \eqref{fj} follows that  $Pf_0=0$ and $Pf_j=-f_{j-1},$ for $j\geq 1$. We will split into two cases our analysis of $P^k(f_j)$, namely, $j-k\geq 0$ and $j-k<0$. First, suppose $j-k\geq 0$. In this case,
\begin{align*}
P^k(f_j)=(-1)^1P^{k-1}(f_{j-1})=(-1)^2P^{k-2}(f_{j-2})=~\cdots=(-1)^{k-1}P^{k-(k-1)}\left(f_{j-(k-1)}\right),
\end{align*}
or equivalently, 
\begin{align*}
P^k(f_j)=(-1)^{k-1}P(f_{j-k+1})=(-1)^{k-1}(-f_{j-k})=(-1)^kf_{j-k}.
\end{align*}
Secondly, assuming $j-k<0$ we have $k-j>0$ so $k-j\geq 1$. Hence
\begin{align*}
P^k(f_j)=(-1)^1P^{k-1}(f_{j-1})=(-1)^2P^{k-2}(f_{j-2})=~\cdots
%&=(-1)^jP^{k-j}(f_{j-j})\\
%&=(-1)^jP^{k-j-1}(Pf_0)\\
=(-1)^jP^{k-j-1}(0)=0.
\end{align*}
Putting both information together we get that
\begin{align}\label{P^kf_j}
P^k(f_j)=\left\{
\begin{array}{ll}
(-1)^kf_{j-k},&\text{ if }j-k\geq 0.\\
0,&\text{ if }j-k<0,
\end{array}
\right.
\end{align}
and consequently
\begin{align}\label{eq 32a}
\partial_t^mP^k\left(f_j(x)y^{(j)}(t)\right)=\left\{
\begin{array}{ll}
(-1)^kf_{j-k}(x)y^{(j+m)}(t),&\text{ \ if \ }j-k\geq 0,\\
0,&\text{ \ if \ }j-k<0.
\end{array}
\right.
\end{align}

Assume $j\geq k$, that is, $j-k\geq 0$. Then thanks to the relation  \eqref{eq 32a} and Lemma \ref{estimetes to f_j and g_j}, the following estimate holds
\begin{align*}
\left|\partial_t^mP^k\left(f_j(x)y^{(j)}(t)\right)\right|\leq 2^{j-k}\frac{|x|^{5(j-k)+1}}{[5(j-k)+1]!}M\frac{(j+m)!^s}{R^{j+m}}\leq M2^{j-k}\frac{(j+m)!^s}{R^{j+m}}\frac{1}{\left(5(j-k)+1\right)!}.
\end{align*}
Setting $l=j-k$ we can write the previous inequality as follows
\begin{align*}
\left|\partial_t^mP^k\left(f_j(x)y^{(j)}(t)\right)\right|&\leq M2^l\frac{(l+k+m)!^s}{R^{l+k+m}}\frac{1}{(5l+1)!}.
\end{align*}
Furthermore, writing $N=k+m$ yields that
\begin{align*}
\left|\partial_t^mP^k\left(f_j(x)y^{(j)}(t)\right)\right|&\leq M2^l\frac{(l+N)!^s}{R^{l+N}}\frac{1}{(5l+1)!}.
\end{align*}
Then using the inequality \eqref{inequality} we obtain that 
\begin{align}\label{eq 33a}
\left|\partial_t^mP^k\left(f_j(x)y^{(j)}(t)\right)\right|&\leq M2^l\frac{(2^l2^Nl!N!)^s}{R^lR^N}\frac{1}{(5l+1)!}=M\frac{2^{(1+s)l}2^{Ns}l!^sN!^s}{R^lR^N}\frac{1}{(5l+1)!}.
\end{align}
Stirling's formula gives that
\begin{align*}
l!\sim \left(\frac{l}{e}\right)^l\sqrt{2\pi l}\text{ \ \ and \ \ }(5l)!\sim \left(\frac{5l}{e}\right)^{5l}\sqrt{2\pi 5l}.
\end{align*}
Observer that
\begin{align*}
\left(\frac{5l}{e}\right)^{5l}\sqrt{2\pi 5l}
%&=5^{5l}\left(\frac{l}{e}\right)^{5l}\sqrt{2\pi l}\sqrt{5}\\
%&=5^{5l+\frac{1}{2}}\left[\left(\frac{l}{e}\right)^l\right]^5\left(\sqrt{2\pi l}\right)\\
=5^{5l+\frac{1}{2}}\left[\left(\frac{l}{e}\right)^l\left(\sqrt{2\pi l}\right)\right]^5\left(\sqrt{2\pi l}\right)^{-4}
\sim 5^{5l+\frac{1}{2}}\left(\sqrt{2\pi l}\right)^{-4}l!^5
\end{align*}
and, consequently, $(5l)!\sim 5^{5l+\frac{1}{2}}\left(\sqrt{2\pi l}\right)^{-4}l!^5.$  Thus, from \eqref{eq 33a}, we get the following
\begin{align*}
\left|\partial_t^mP^k\left(f_j(x)y^{(j)}(t)\right)\right|
%&\leq M\frac{l!^s}{\left[2^{-(1+s)}R\right]^l}\frac{N!^s}{\left(2^{-s}R\right)^N}\frac{1}{(5l+1)(5l)!}\\
&\leq \tilde{M}\frac{l!^s}{\left[2^{-(1+s)}R\right]^l}\frac{N!^s}{\left(2^{-s}R\right)^N}\frac{1}{(5l+1)5^{5l+\frac{1}{2}}\left(\sqrt{2\pi l},\right)^{-4}l!^5},
\end{align*}
for a suitable constant $\tilde{M}>0$. Since $\frac{1}{(5l+1)5^{5l+\frac{1}{2}}}\leq 1,$ for $l\geq 0$, it follows that
$$
\left|\partial_t^m P^k\left(f_j(x)y^{(j)}(t)\right)\right|\leq \tilde{M}\frac{N!^s}{\left(2^{-s}R\right)^N}\frac{(2\pi l)^2}{[2^{-(1+s)}R]^ll!^{5-s}}.
$$
Using the inequality \eqref{inequality} we get that
\begin{align*}
\frac{N!^s}{\left(2^{-s}R\right)^N}\leq \frac{k!^s}{\left(\frac{R}{4^s}\right)^k}\frac{m!^s}{\left(\frac{R}{4^s}\right)^m}
\end{align*}
and, consequently, the following holds
\begin{align*}
\left|\partial_t^m P^k\left(f_j(x)y^{(j)}(t)\right)\right|\leq \tilde{M} 4 \pi^2 \frac{k!^s}{\left(4^{-s}R\right)^k}\frac{m!^s}{\left(4^{-s}R\right)^m}\frac{l^2}{\tilde{R}^ll!^{5-s}},
\end{align*}
where $\tilde{R}=2^{-(1+s)}R$. 
%Therefore
%\begin{align*}
%\left\|\partial_t^m P^k\left(f_jy^{(j)}(t)\right)\right\|_\infty\leq 4\tilde{M}  \pi^2 \frac{k!^s}{\left(4^{-s}R\right)^k}\frac{m!^s}{\left(4^{-s}R\right)^m}\frac{l^2}{\tilde{R}^ll!^{5-s}},\ \ \forall m,k\geq 0,\ j\geq k,\ l=j-k
%\end{align*}
%and consequently
Therefore, 
\begin{align*}
\sum_{j\geq k}\left\|\partial_t^m P^k\left(f_jy^{(j)}(t)\right)\right\|_\infty\leq 4\tilde{M}  \pi^2 \frac{k!^s}{\left(4^{-s}R\right)^k}\frac{m!^s}{\left(4^{-s}R\right)^m}\sum_{l\geq 0}\frac{l^2}{\tilde{R}^ll!^{5-s}},
\end{align*}
for all $m,k\geq 0$. To finish this case, observe that the following series $\sum_{l\geq 0}\frac{l^2}{\tilde{R}^ll!^{5-s}}$ is convergent, so by the Weierstrass M-test the following series
\begin{align*}
\sum_{j\geq k}\left\|\partial_t^m P^k\left(f_jy^{(j)}(t)\right)\right\|_\infty
\end{align*}
is uniformly convergent on $[-1,0]\times [0,T]$, for all $m,k\geq 0$. Furthermore
\begin{align}\label{eq 35a}
	\sum_{j\geq k}\left\|\partial_t^m P^k\left(f_jy^{(j)}(t)\right)\right\|_\infty\leq \overline{M}\frac{k!^s}{\left(4^{-s}R\right)^k}\frac{m!^s}{\left(4^{-s}R\right)^m},
\end{align}
for $m,k\geq 0$, where
$
\overline{M}=4\tilde{M}  \pi^2 \sum_{l\geq 0}\frac{l^2}{\tilde{R}^ll!^{5-s}}.
$

Analogously it turns out that
\begin{align}\label{P^kg_j}
	P^k(g_j)=\left\{
	\begin{array}{ll}
		(-1)^kg_{j-k},&\text{ if }j-k\geq 0,\\
		0,&\text{ if }j-k<0,
	\end{array}
	\right.
\end{align}
and
\begin{align}\label{eq 36a}
	\partial_t^mP^k\left(g_j(x)z^{(j)}(t)\right)=\left\{
	\begin{array}{ll}
		(-1)^kg_{j-k}(x)z^{(j+m)}(t),&\text{ \ if \ }j-k\geq 0,\\
		0,&\text{ \ if \ }j-k<0,
	\end{array}
	\right.
\end{align}
so that, using inequalities \eqref{inequality}, \eqref{eq 30a}, Stirling's formula and Lemma \ref{f_j and g_j} we get
\begin{align*}
	\sum_{j\geq k}\left\|\partial_t^m P^k\left(g_jz^{(j)}(t)\right)\right\|_\infty\leq 4\tilde{M}  \pi^2 \frac{k!^s}{\left(4^{-s}R\right)^k}\frac{m!^s}{\left(4^{-s}R\right)^m}\sum_{l\geq 0}\frac{l^2}{\tilde{R}^ll!^{5-s}},
\end{align*}
for all $m,k\geq 0$. Consequently the series $\sum\left\|\partial_t^m P^k\left(g_jz^{(j)}(t)\right)\right\|_\infty$ is uniformly convergent in $[-1,0]\times [0,T]$, for all $m,k\geq 0$, with
\begin{align}\label{eq 37}
	\sum_{j\geq k}\left\|\partial_t^m P^k\left(g_jz^{(j)}(t)\right)\right\|_\infty\leq \overline{M}\frac{k!^s}{\left(4^{-s}R\right)^k}\frac{m!^s}{\left(4^{-s}R\right)^m}.
\end{align}

Let $K_3>0$ be as  in Lemma \ref{lemma A3} for $p=\infty$. From \eqref{eq 32a} and \eqref{eq 35a} we have
\begin{align*}
	\sum_{j=0}^\infty\left\|\partial_t^m\left(f_jy^{(j)}(t)\right)\right\|_{5i,\infty}= K_3^i\sum_{k=0}^i\sum_{j\geq k} \left\|\partial_t^m P^k \left(f_jy^{(j)}(t)\right)\right\|_\infty\leq K_3^i\sum_{k=0}^i\overline{M}\frac{k!^s}{\left(4^{-s}R\right)^k}\frac{m!^s}{\left(4^{-s}R\right)^m}.
\end{align*}
Thus
\begin{align*}
	\sum_{j=0}^\infty\left\|\partial_t^m\left(f_jy^{(j)}(t)\right)\right\|_{5i,\infty}&\leq K_3^i\overline{M}i!^s\left(\sum_{k=0}^i\frac{1}{\left(4^{-s}R\right)^k}\right)\frac{m!^s}{\left(4^{-s}R\right)^m}.
\end{align*}
Defining $b=\frac{1}{4^{-s}R}$ and noting that
\begin{align*}
	\sum_{j=0}^i\frac{1}{\left(4^{-s}R\right)^k}&=\sum_{k=0}^ib^k=1\cdot\frac{b^{i+1}-1}{b-1}\leq \frac{b^{i+1}}{b-1}=\frac{b}{b-1}b^i,
\end{align*}
yields 
\begin{align*}
	\sum_{j=0}^\infty\left\|\partial_t^m\left(f_jy^{(j)}(t)\right)\right\|_{5i,\infty}&\leq K_3^i\overline{M} i!^s\frac{b}{b-1}b^i\frac{m!^s}{\left(4^{-s}R\right)^m}.
\end{align*}
Hence, defining $\hat{M}=\overline{M}\frac{b}{b-1}$ it follows that
\begin{align}\label{eq 38a}
	\sum_{j=0}^\infty\left\|\partial_t^m\left(f_jy^{(j)}(t)\right)\right\|_{5i,\infty}&\leq \hat{M}\frac{i!^s}{(K_3^{-1}b^{-1})^i}\frac{m!^s}{\left(4^{-s}R\right)^m},\ \forall m,i\geq 0,\ t \in [0,T].
\end{align}
Similarly, from \eqref{eq 36a}, \eqref{eq 37} and Lemma \ref{lemma A3} we obtain
\begin{align}\label{eq 39a}
	\sum_{j=0}^\infty\left\|\partial_t^m\left(g_jz^{(j)}(t)\right)\right\|_{5i,\infty}&\leq \hat{M}\frac{i!^s}{(K_3^{-1}b^{-1})^i}\frac{m!^s}{\left(4^{-s}R\right)^m},\ \forall m,i\geq 0,\ t \in [0,T].
\end{align}

Given $m,n\geq 0$ consider $i\geq 0$ such that $n\in\left\{5i-r,\ r=0,1,2,3,4\right\}.$
Then from \eqref{eq 38a} we get, for $(x,t)\in [-1,0]\times [0,T]$, that
\begin{align*}
	\sum_{j=0}^\infty\left|\partial_x^n\partial_t^m\left(f_j(x)y^{(j)}(t)\right)\right|&\leq \hat{M}\frac{i!^s}{(K_3^{-1}b^{-1})^i}\frac{m!^s}{\left(4^{-s}R\right)^m},
\end{align*}
and from \eqref{eq 39a}
\begin{align*}
	\sum_{j=0}^\infty\left|\partial_x^n\partial_t^m\left(g_j(x)z^{(j)}(t)\right)\right|&\leq \hat{M}\frac{i!^s}{(K_3^{-1}b^{-1})^i}\frac{m!^s}{\left(4^{-s}R\right)^m}.
\end{align*}
By the Weierstrass M-test, it follows these series are uniformly convergent in $[-1,0]\times [0,T]$ for all $m,n\geq 0$. Thus, defining $u:[-1,0]\times [0,T]\rightarrow \mathbb{R}$ as in \eqref{flatness solution} we have that $\partial_x^n\partial_t^mu$ is continuous. Consequently $u \in C^\infty([-1,0]\times [0,T])$ and satisfies
\begin{align}\label{eq 40a}
	\left|\partial_x^n\partial_t^mu(x,t)\right|\leq 2\hat{M}\frac{i!^s}{(K_3^{-1}b^{-1})^i}\frac{m!^s}{\left(4^{-s}R\right)^m}\ \ \forall (x,t)\in [-1,0]\times [0,T].
\end{align}
As seen before, Stirling's formula gives us  $(5i)!\sim 5^{5i+\frac{1}{2}}\left(\sqrt{2\pi i}\right)^{-4}i!^5=5^{5i+\frac{1}{2}}\left(2\pi i\right)^{-2}i!^5$, and so that $i!^5\sim \frac{(2\pi i)^2(5i)!}{5^{5i+\frac{1}{2}}}$. Therefore $i!^s\sim \left[\frac{(2\pi i)^2(5i)!}{5^{5i+\frac{1}{2}}}\right]^\frac{s}{5}.$ Once we have that  $\frac{i^2}{5^{5i+\frac{1}{2}}}\leq 1\ \ \forall i\geq 0$, we can infer that $i!^s\leq M^*(4\pi^2)^\frac{s}{5}(5i)!^\frac{s}{5},$ for some constant $M^*>0$. Furthermore, $n=5i-r$, that is, $5i=n+r$ with $r \in \{0,1,2,3,4\}$. Then, by \eqref{inequality},
\begin{align*}
	i!^s\leq M^*\left(4\pi^2\right)^\frac{s}{5}\left(2^n2^rn!r!\right)^\frac{s}{5}=&M^*\left(4\pi^2\cdot 2^r\cdot r!\right)^\frac{s}{5}(2^n)^\frac{s}{5}n!^\frac{s}{5}	\leq M'\left(2^\frac{s}{5}\right)^nn!^\frac{s}{5},
\end{align*}
where $M'=M^*\left(4\pi^2\cdot 2^r\cdot r!\right)^\frac{s}{5}$. With this, returning to inequality \eqref{eq 40a} we obtain
\begin{align*}
	\left|\partial_x^n\partial_t^mu(x,t)\right|&\leq 2\hat{M}M'\left(2^\frac{s}{5}\right)^n\frac{n!^\frac{s}{5}}{\left(K_3^{-1}b^{-1}\right)^\frac{n+r}{5}}\frac{m!^s}{\left(4^{-s}R\right)^m}\leq \frac{2\hat{M}M'}{\left(K_3^{-1}b^{-1}\right)^\frac{r}{5}}\frac{n!^\frac{s}{5}}{\left(K_3^{-\frac{1}{5}}b^{-\frac{1}{5}}2^{-\frac{s}{5}}\right)^n}\frac{m!^s}{\left(4^{-s}R\right)^m}.
\end{align*}
Finally, defining
\begin{align*}
	M'':=\max\left\{\frac{2\hat{M}M'}{\left(K_3^{-1}b^{-1}\right)^\frac{r}{5}};\ r=0,1,2,3,4\right\},\ \ \ R_1:=K_3^{-\frac{1}{5}}b^{-\frac{1}{5}}2^{-\frac{s}{5}},\ \text{and} \ R_2:=4^{-s}R,
\end{align*}
we have
\begin{align*}
	\left|\partial_x^n\partial_t^mu(x,t)\right|&\leq M''\frac{n!^\frac{s}{5}}{R_1^n}\frac{m!^s}{R_2^m},\ \ \forall n,m\geq 0,\ \forall (x,t)\in [-1,0]\times [0,T].
\end{align*}
Moreover, $u$ solves \eqref{flatness system} by construction, the result follows. 
\end{proof}

\subsection{Estimates in $W^{5n,p}(-1,0)$-norm} 
Let us starting remembering that the map
\begin{align*}
\begin{array}{rcl}
	\|\cdot\|_*:W^{5,p}(-1,0)&\rightarrow&\mathbb{R}_+\\
	f&\mapsto&\|f\|_*:=\|f\|_p+\|Pf\|_p
\end{array}
\end{align*}
is a norm called the \textit{graph} norm associated with the operator $P:W^{5,p}(-1,0)\rightarrow L^p(-1,0)$.
\begin{lemma}\label{lemma A1}
Let $p\in[1,\infty]$ be. For all $n\geq 0$ we have
\begin{align}\label{eq 77}
\|P^nf\|_p\leq 3^n\|f\|_{5n,p},\ \forall\ f \in W^{5n,p}(-1,0).
\end{align}
\end{lemma}
\begin{proof}
For $n=0$ the inequality is obvious. For $n=1$, given $f \in W^{5n,p}(-1,0)$ follows that
\begin{align*}
\|Pf\|_p&=\|\partial_x f+\partial_x^3f-\partial_x^5f\|_p\leq \|\partial_x f\|_p+\|\partial_x^3f\|_p+\|\partial_x^5f\|_p\leq 3\|f\|_{5,p}.
\end{align*}
Suppose that \eqref{eq 77} for $0,1,...,n-1$. Then for $f \in W^{5n,p}(-1,0)$ we get
\begin{align*}
\|P^nf\|_p=\|P^{n-1}Pf\|_p&\leq 3^{n-1}\|Pf\|_{5n-5,p}
%\\
%&\leq 3^{n-1}\left(\|\partial_xf\|_{5n-5,p}+\|\partial_x^3f\|_{5n-5,p}+\|\partial_x^5f\|_{5n-5,p}\right)\\
%&\leq 3^{n-1}\left(\sum_{i=0}^{5n-5}\|\partial_x^{i+1}f\|_p+\sum_{i=0}^{5n-5}\|\partial_x^{i+3}f\|_p+\sum_{i=0}^{5n-5}\|\partial_x^{i+5}f\|_p\right)\\
%\leq 3^{n-1}\left(\sum_{j=1}^{5n-4}\|\partial_x^{j}f\|_p+\sum_{j=3}^{5n-2}\|\partial_x^{j}f\|_p+\sum_{j=5}^{5n}\|\partial_x^{j}f\|_p\right)
\leq 3^{n-1}\left(3\sum_{j=0}^{5n}\|\partial_x^{j}f\|_p\right)=3^n\|f\|_{5n,p},
\end{align*}
so \eqref{eq 77} is true for $n$, which concludes the proof.
\end{proof}
\begin{lemma}\label{lemma A2}
Let $p \in [1,\infty]$ be. There exists a constant $K_1=K_1(p)>0$ such that
\begin{align*}
\|f\|_{5,p}\leq K_1\left(\|f\|_p+\|Pf\|_p\right),\ \forall\ f \in W^{5,p}(-1,0).
\end{align*}
\end{lemma}
\begin{proof}
We know that $\|\cdot\|_*$ is a norm in $W^{5,p}(-1,0)$. We will show that $\left(W^{5,p}(-1,0),\|\cdot\|_*\right)$ is a Banach space. To do this, consider a Cauchy sequence in $\left(W^{5,p}(-1,0),\|\cdot\|_*\right)$, $(f_n)_{n\geq 0}$. Then,
\begin{align*}
\|f_m-f_n\|_*=\|f_m-f_n\|_p+\|Pf_m-Pf_n\|_p\rightarrow 0 \text{ \ as \ }m,n\rightarrow\infty.
\end{align*}
Since $L^p(-1,0)$ is a Banach space, there exist $f,g \in L^p(-1,0)$ such that
\begin{align*}
\|f_n-f\|_p\rightarrow 0,\ \ \ \|Pf_n-g\|_p\rightarrow 0\text{ \ as \ }n\rightarrow \infty.
\end{align*}

Given $\varphi \in C_0^\infty(-1,0)$ we have
\begin{align*}
\left|\int_{-1}^0(f_n-f)\varphi\right|\leq \int_{-1}^0|f_n-f||\varphi|\leq \left\{
\begin{array}{ll}
\|f_n-f\|_p\|\varphi\|_q,&1<p<\infty,\ \frac{1}{p}+\frac{1}{q}=1\\
\|f_n-f\|_1\|\varphi\|_\infty,&p=1\\
\|f_n-f\|_\infty\|\varphi\|_1,&p=\infty.
\end{array}
\right.
\end{align*}
Therefore
\begin{align*}
\int_{-1}^0(f_n-f)\varphi\rightarrow 0,\ \ \forall \varphi\in C_0^\infty(-1,0)
\end{align*}
which implies that 
\begin{align}\label{eq 78}
f_n\rightarrow f \text{ \ in \ }\mathcal{D}'(-1,0).
\end{align}

Analogously we infer that $Pf_n\rightarrow g$ in $\mathcal{D}'(-1,0)$. But  \eqref{eq 78} implies that
\begin{align*}
\partial_x^if_n\rightarrow\partial_x^if \text{ \ in \ }\mathcal{D}'(-1,0),\ \forall i \in \mathbb{N}\cup\{0\}
\end{align*}
and consequently $Pf_n\rightarrow Pf$ in $\mathcal{D}'(-1,0)$. By uniqueness of limit, it follows that
\begin{align}\label{eq 79}
Pf=g\text{ \ in \ }\mathcal{D}'(-1,0).
\end{align}

Consider $T_1,T_2\in D'(-1,0)$ given by $T_1=f+\partial_x^2f-\partial_x^4f$ and $T_2=h_1$, where $$h_1(x)=\int_{-1}^xg(t)dt.$$ Note that \eqref{eq 79} implies that  $\partial_xT_1=g$ in $\mathcal{D}'(-1,0).$ On the other hand, $h_1\in L^p(-1,0)$ and
\begin{align*}
\left\langle \partial_xT_2,\varphi\right\rangle_{\mathcal{D}',\mathcal{D}}&=-\int_{-1}^0h_1(x)\varphi'(x)dx=-\int_{-1}^0\left(\int_{-1}^xg(t)dt\right)\varphi'(x)dx=-\int_{-1}^0\int_{-1}^xg(t)\varphi'(x)dtdx.
\end{align*}
For $-1\leq t\leq x\leq 0$ the Fubini's theorem gives us
\begin{align*}
	\left\langle \partial_xT_2,\varphi\right\rangle_{\mathcal{D}',\mathcal{D}}
	%&=-\int_{-1}^0\int_{t}^0g(t)\varphi'(x)dxdt=-\int_{-1}^0g(t)\int_t^0\varphi'(x)dxdt\\
	=-\int_{-1}^0g(t)(\varphi(0)-\varphi(t))dt=\int_{-1}^0g(t)\varphi(t)dt
	=\left\langle g,\varphi \right\rangle_{\mathcal{D}',\mathcal{D}}
\end{align*}
and therefore $\partial_xT_2=g=\partial_xT_1$ in $\mathcal{D}'(-1,0).$ Consequently, there exists a constant $k_1\in \mathbb{R}$ such that $T_1=T_2+k_1$ in $\mathcal{D}'(-1,0),$ that is,
\begin{align*}
f+\partial_x^2f-\partial_x^4f=h_1+k_1\text{ \ in \ } \mathcal{D}'(-1,0).
\end{align*}
Defining $\tilde{h}_1=h_1+k_1-f$ we have $\tilde{h}_1 \in L^p(-1,0)$ and
\begin{align}\label{eq 80}
\partial_x^2f-\partial_x^4f=\tilde{h}_1\text{ \ in \ }\mathcal{D}'(-1,0).
\end{align}

Now, consider $T_3,T_4 \in \mathcal{D}'(-1,0)$ given by $T_3=\partial_xf-\partial_x^3f$ and $T_4=h_2$, where $h_2 \in L^p(-1,0)$ is given by
\begin{align*}
h_2(x)=\int_1^x\tilde{h}_1(t)dt.
\end{align*}
Proceeding as done before we obtain  $\partial_xT_3=\tilde{h}_1=\partial_xT_4$ in $\mathcal{D}'(-1,0)$, thus there exists a constant $k_2\in \mathbb{R}$ such that $T_3=T_4+k_2\text{ \ in \ }\mathcal{D}'(-1,0)$, or equivalently,
\begin{align*}
\partial_xf-\partial_x^3f=h_2+k_2\text{ \ in \ }\mathcal{D}'(-1,0).
\end{align*}
Defining $\tilde{h}_2:=h_2+k_2$ we have $\tilde{h}_2 \in L^p(-1,0)$ and
\begin{align}\label{eq 81}
\partial_xf-\partial_x^3f=\tilde{h}_2\text{ \ in \ }\mathcal{D}'(-1,0).
\end{align}

Set $T_5,T_6 \in \mathcal{D}'(-1,0)$ by $T_5=f-\partial_x^2f$ and $T_6=h_3$, where $h_3\in L^p(-1,0)$ is given by
\begin{align*}
h_3(x)=\int_{-1}^x\tilde{h}_2(t)dt.
\end{align*}
By the same argument as done before,
\begin{align*}
f-\partial_x^2f=h_3+k_3\text{ \ in \ }\mathcal{D}'(-1,0)
\end{align*}
with $k_3\in \mathbb{R}$. Defining $\tilde{h}_3=f-h_3-k_3$ we have $\tilde{h}_3\in L^p(-1,0)$ and
\begin{align}\label{eq 82}
\partial_x^2f=\tilde{h}_3\text{ \ in \ }\mathcal{D}'(-1,0).
\end{align}

Now define $T_7, T_8\in \mathcal{D}'(-1,0)$ by $T_7=\partial_xf$ and $T_8=h_4$, where $h_4 \in L^p(-1,0)$ is given by
\begin{align*}
h_4(x)=\int_{-1}^t\tilde{h}(t)dt.
\end{align*}
Again, the same argument ensures that there exists $k_4\in\mathbb{R}$ such that, defining $\tilde{h}_4:=h_4+k_4$ we have $\tilde{h}_4\in L^p(-1,0)$ and
\begin{align}\label{eq 83}
\partial_x f=\tilde{h}_4\text{ \ in \ }\mathcal{D}'(-1,0).
\end{align}

Since $\tilde{h}_1\tilde{h}_2,\tilde{h}_3,\tilde{h}_4\in L^p(-1,0)$, the equalities \eqref{eq 79}-\eqref{eq 83} gives us $
\partial_xf,\partial_x^2f,\partial_x^3f,\partial_x^4f$ and $\partial_x^5f$ belong to $L^p(-1,0)$, so that $f \in W^{5,p}(-1,0)$. Furthermore,
\begin{align*}
\|f_n-f\|_*=\|f_n-f\|_p+\|Pf_n-Pf\|_p=\|f_n-f\|_p+\|Pf_n-g\|_p
\end{align*}
which implies that
\begin{align*}
\|f_n-f\|_*\rightarrow 0\text{ \ as \ }n\rightarrow\infty,
\end{align*}
that is, $f_n\rightarrow f$ in $\left(W^{5,p}(-1,0),\|\cdot\|_*\right)$ showing that $\left(W^{5,p}(-1,0),\|\cdot\|_*\right)$ is a Banach space.

Consider the map
\begin{align*}
	\begin{array}{rcl}
		I:\left(W^{5,p}(-1,0),\|\cdot\|_{5,p}\right)&\rightarrow&\left(W^{5,p}(-1,0),\|\cdot\|_*\right)\\
		f&\mapsto&I(f)=f.
	\end{array}
\end{align*}
Note that $I$ is linear and bijective and $\|I(f)\|_*\leq \|f\|_{5,p}$, so that $I$ is continuous. Thus, $I^{-1}$ is also continuous. Therefore, there exists $K_1>0$ such that
\begin{align*}
\|f\|_{5,p}\leq K_1\left(\|f\|_p+\|Pf\|_p\right)\ \ \ \forall f \in W^{5,p}(-1,0),
\end{align*}
showing the result.
\end{proof}
\begin{remark}\label{remark A1}
As a consequence of the Lemma \ref{lemma A2}, we see that the norms $\|\cdot\|_*$ and $\|\cdot\|_{5,p}$ are equivalents in the space $W^{5,p}(-1,0)$.
\end{remark}
\begin{lemma}\label{lemma A3}
Let $p \in [1,\infty]$ be. There exists a constant $K_2=K_2(p)>0$ such that, for every $n\geq 1$,
\begin{align*}
2\cdot 3^{-(n+1)}\sum_{i=0}^n\|P^if\|_p\leq \|f\|_{5n,p}\leq (1+2K_2)^{n-1}K_2\sum_{i=0}^n\|P^if\|_p\leq \|f\|_{5n,p},\ \forall\ f \in W^{5n,p}(-1,0).
\end{align*}
Consequently,
\begin{align*}
	2\cdot3^{-(n+1)}\sum_{i=0}^n\|P^if\|_p\leq \|f\|_{5n,p}\leq K_3^n\sum_{i=0}^n\|P^if\|_p,\ \forall\ f \in W^{5n,p}(-1,0),
\end{align*}
 for every $n\geq 0$, where $K_3=K_3(p)=1+2K_2$.
\end{lemma}

\begin{proof}
Given $n\geq 0$ and $f \in W^{5n,p}(-1,0)$ we have $f \in W^{5i,p}(-1,0)$ and $\|f\|_{5i,p}\leq \|f\|_{5n,p},$ $0\leq i\leq n.$ Then by Lemma \ref{lemma A1} we get
\begin{align*}
	\sum_{i=0}^n\|P^if\|_p&\leq \sum_{i=0}^n 3^i\|f\|_{5i,p}\leq \left(\sum_{i=0}^n3^i\right)\|f\|_{5n,p}=1\cdot \frac{3^{n+1}-1}{3-1}\|f\|_{5n,p}\leq \frac{3^{n+1}}{2}\|f\|_{5n,p}
\end{align*}
and therefore
\begin{align*}
2\cdot3^{-(n+1)}\sum_{i=0}^n\|P^if\|_p\leq \|f\|_{5n,p},\ \ \forall f \in W^{5n,p}(-1,0),\ \ \forall n\geq 0.
\end{align*}
To prove
\begin{align*}
\|f\|_{5n,p}\leq (1+2K_1)^{n-1}K_1\sum_{i=0}^n\|P^if\|_p\ \ \forall f \in W^{5n,p}(-1,0),\ \ \forall n\geq 1
\end{align*}
we use induction on $n$. For $n=1$ this is true since, by Lemma \ref{lemma A2},
\begin{align*}
\|f\|_{5,p}\leq K_1\left(\|f\|_p+\|Pf\|_p\right)=K_1\sum_{i=0}^1\|P^if\|_p=(1+2K_1)^0K_1\sum_{i=0}^1\|P^if\|_p.
\end{align*}
Assume that, for $n\geq 2$, the inequality is true up to the rank $n-1$. Then, for any $f\in W^{5n,p}(-1,0)$ we have (setting $j=i-5n+5$)
\begin{align*}
\|f\|_{5n,p}%&=\|f\|_{5n-5,p}+\sum_{i=5n-4}^{5n}\|\partial_x^if\|_p=\|f\|_{5n-5,p}+\sum_{j=1}^5\|\partial_x^{j+5n-5}f\|_p\\
&=\|f\|_{5n-5,p}+\sum_{j=1}^5\|\partial_x^j\partial_x^{5n-5}f\|_p\leq \|f\|_{5n-5,p}+\|\partial_x^{5n-5}f\|_{5,p}.
\end{align*}
Using Lemma \ref{lemma A2} and the induction hypothesis we get
\begin{align*}
\|f\|_{5n,p}&\leq (1+2K_1)^{(n-1)-1}K_1\sum_{i=0}^{n-1}\|P^if\|_p
+K_1\left(\|\partial_x^{5n-5}f\|_p+\|P\partial_x^{5n-5}f\|_p\right)\\
&\leq (1+2K_1)^{n-2}K_1\sum_{i=0}^{n-1}\|P^if\|_p+K_1\left(\|f\|_{5n-5,p}+\|Pf\|_{5n-5,p}\right).
\end{align*}
Once again induction hypothesis yields
\begin{align*}
\|f\|_{5n,p}\leq&(1+2K_1)^{n-2}K_1\sum_{i=0}^{n-1}\|P^if\|_p\\
&+K_1\left((1+2K_1)^{n-2}K_1\sum_{i=0}^{n-1}\|P^if\|_p+(1+2K_1)^{n-2}K_1\sum_{i=0}^{n-1}\|P^iPf\|_p\right).
\end{align*}
Then
\begin{align*}
\|f\|_{5n,p}
%&\leq(1+2K_1)^{n-2}K_1\sum_{i=0}^{n-1}\|P^if\|_p+K_1(1+2K_1)^{n-2}K_1\left(\sum_{i=0}^{n-1}\|P^if\|_p+\sum_{i=1}^{n}\|P^if\|_p\right)\\
&\leq (1+2K_1)^{n-2}K_1\sum_{i=0}^{n}\|P^if\|_p+2K_1(1+2K_1)^{n-2}K_1\sum_{i=0}^{n}\|P^if\|_p\\
%&=(1+2K_1)(1+2K_1)^{n-2}K_1\sum_{i=0}^{n}\|P^if\|_p\\
&=(1+2K_1)^{n-1}K_1\sum_{i=0}^{n}\|P^if\|_p,
\end{align*}
from where the desired follows with $K_2=K_1$.

Finally, consider $K_3=1+2K_2$ and $f \in W^{5n,p}(-1,0)$. For $n=0$, we see that  $\|f\|_{5n,p}=K_3^n\sum_{i=0}^{n}\|P^if\|_p.$ Now, For $n\geq 1$, we have that
\begin{align*}
\|f\|_{5n,p}&\leq (1+2K_2)^{n-1}K_2\sum_{i=0}^{n}\|P^if\|_p\leq (1+2K_2)^{n-1}(1+2K_2)\sum_{i=0}^{n}\|P^if\|_p=K_3^n\sum_{i=0}^{n}\|P^if\|_p.
\end{align*}
\end{proof}

\subsection{Smoothing property} With the previous section in hand, let us show that for any $u_0 \in L^2(-1,0)$, the solution of \eqref{kawahara control system} with $h_1=h_2=0$ is a Gevrey unction of order $\frac{1}{2}$ in the variable $x$ and $\frac{5}{2}$ in the variable $t$. Precisely, we will prove the existence of positive constants $M,R_1$, and $R_2$ such that
\begin{align}\label{eq 23}
\left|\partial_x^p\partial_t^nu(x,t)\right|\leq \frac{M}{t^\frac{5n+p+5}{2}}\frac{p!^{\frac{1}{2}}}{R_1^p}\frac{n!^{\frac{5}{2}}}{R_2^n},\end{align}
for all $t\in(0,T]$ and for all $x \in [-1,0].$

To do this, using the inequality \eqref{eq 23} on intervals of length one we can assume, without loss of generality, $T=1$. Consider the operator given by $Au=-Pu=-\partial_xu-\partial_x^3u+\partial_x^5u,$ with $D(A)=\left\{u \in H^5(-1,0);\ u(-1)=u(0)=u_x(-1)=u_x(0)=u_{xx}(0)=0\right\}.$  It has been proven in \cite{Araruna Capistrano Doronin} that $A$ generates a semigroup of contractions $\{S(t)\}_{t\geq 0}$ in $L^2(-1,0)$, moreover, that smoothing effect is verified, that is, for $u_0 \in L^2(-1,0)$, the mild solution $u(\cdot,t)=S(t)u_0$ of \eqref{kawahara control system} with $h_1=h_2=0$ satisfies
$
u \in C([0,1],L^2(-1,0))\cap L^2(0,1,H^2(-1,0))
$
and
\begin{align}\label{norm of u in L2(H2)}
\|u\|_{L^2(0,T,H^2(-1,0))}\leq \sqrt{3}\|u_0\|_{L^2(-1,0)}.
\end{align}

Now on, we will denote the norm $\|\cdot\|_{L^2(-1,0)}$ by simplicity for $\|\cdot\|_{L^2}$ and the spaces
\begin{align*}
X_0&=L^2(-1,0)\quad X_1=H_0^1(-1,0),\quad X_2=H_0^2(-1,0),\\
X_3&=\left\{u \in H^3(-1,0);\ u(-1)=u(0)=u_x(-1)=u_x(0)=u_{xx}(0)=0\right\},\\
X_4&=\left\{u \in H^4(-1,0);\ u(-1)=u(0)=u_x(-1)=u_x(0)=u_{xx}(0)=0\right\},\\
X_5&=D(A).
\end{align*}
For any $m \in \{1,2,3,4,5\}, (X_m,\|\cdot\|_{H^m})$ is a Banach space.
% endowed with the following norm
%\begin{align*}
%\|u\|_{H^m}=\left(\sum_{i=0}^m\left\|\partial_x^iu\right\|_{L^2}^2\right)^\frac{1}{2}.
%\end{align*}
The next propositions are paramount in our analysis and given several estimates in the $X_s$-spaces. 
\begin{proposition}\label{S:L^2->H^2}
For any $t \in (0,1]$ the map $S(t):L^2(-1,0)\rightarrow H^2(-1,0)$ is continuous. More precisely, there exists a positive constant $C_1>0$ such that $\|u(\cdot,t)\|_{H^2}\leq C_1\|u_0\|_{L^2},$ where $u(\cdot,t)=S(t)u_0$.
\end{proposition}
\begin{proof}
Given $u_0 \in D(A)$, the semigroup theory ensures that  $S(t)u_0 \in D(A)$ with
\begin{align}\label{eq 25}
\frac{d}{dt}S(t)u_0=AS(t)u_0=S(t)Au_0
\end{align}
and  $u=S(\cdot)u_0 \in C\left([0,\infty),D(A)\right).$ From \eqref{eq 25} we obtain $\|AS(t)u_0\|_{L^2}=\|S(t)Au_0\|_{L^2}$ and, therefore, $\|S(t)u_0\|_{D(A)}\leq\|u_0\|_{D(A)}.$

Using the Lemma \ref{lemma A2} (see also Remark \ref{remark A1}) we obtain a positive constant $C'>0$ (which does not depend on $t$) such that $\|S(t)u_0\|_{H^5}\leq C_1'\|u_0\|_{H^5}.$  Then we see that the map $S(t):\left(D(A),\|\cdot\|_{H^5}\right)\to \left(D(A),\|\cdot\|_{H^5}\right)$ is continuous. Since $S(t):L^2(-1,0)\to L^2(-1,0)$ is also continuous and, by interpolation argument, $\left[X_0,X_5\right]_{\frac{2}{5}}=X_2$, we have that  $S(t):\left(X_2,\|\cdot\|_{H^2}\right)\to \left(X_2,\|\cdot\|_{H^2}\right)$ is continuous with $
\|u(\cdot,t)\|_{H^2}=\|S(t)u_0\|_{H^2}\leq C_1''\|u_0\|_{H^2},$ for $u_0 \in X_2$ and $t \in (0,1]$, where $C_1''=\max\{C_1',1\}$.

In this way, given $u_0 \in X_2$, for any $t \in(0,1]$ we have
\begin{align*}
\|u(\cdot,t)\|_{H^2}^2\leq (C_1'')^2\|u(\cdot,s)\|_{H^2}^2,
\end{align*}
for all $s \in (0,t].$ Integrating with respect to $s$ from $0$ to $t$ we get
\begin{align*}
t\|u(\cdot,t)\|_{H^2}^2\leq (C_1'')^2\int_0^t\|u(\cdot,s)\|_{H^2}^2ds\leq (C_1'')^2\|u\|_{L^2(0,1,H^2(-1,0))}^2.
\end{align*}
Using \eqref{norm of u in L2(H2)} we obtain
\begin{align*}
\|u(\cdot,t)\|_{H^2}\leq \frac{C_1''\sqrt{3}}{\sqrt{t}}\|u_0\|_{L^2}^2,
\end{align*}
and the result is achieved.
\end{proof}
\begin{proposition}\label{S:H^5->H^7}
For any $t \in (0,1]$ the map $S(t):D(A)\rightarrow H^7(-1,0)$ is continuous, that is, there exists a positive constant $C_2>0$ such that $\|u(\cdot,t)\|_{H^7}\leq C_2\|u_0\|_{H^5},$ for all $u_0 \in D(A)$.
\end{proposition}
\begin{proof}
First, we need to check that $S(t)u_0 \in H^7(-1,0)$ whenever $u_0 \in D(A)$. Indeed, given $u_0 \in D(A)$ we have $u(\cdot,t)=S(t)u_0 \in D(A)$ and \eqref{eq 25} holds. Then $Au(\cdot,t)=AS(t)u_0=S(t)Au_0 \in H^2(-1,0)$ which implies that $\partial_x^5u(\cdot,t)\in H^2(-1,0)$ and consequently $u(\cdot,t)\in H^7(-1,0).$ Furthermore, Proposition \ref{S:L^2->H^2}  gives us
\begin{align}\label{eq 26}
\|u(\cdot,t)\|_{H^2}+\left\|Au(\cdot,t)\right\|_{H^2}\leq  \frac{C_1}{\sqrt{t}}\left(\|u_0\|_{L^2}+\|Au_0\|_{L^2}\right).
\end{align}

Now, let $v \in H^7(-1,0)$. Observe that $\partial_x^6v=\partial_x^2 v+\partial_x^4v-\partial_xPv$ and $\partial_x^7v=\partial_x^3 v+\partial_x^5v-\partial_x^2Pv.$ Then, 
\begin{align*}
\|v\|_{H^7}\leq C_2'\left(\|v\|_{H^5}+\|Pv\|_{H^2}\right),
\end{align*}
for some positive constant $C_2'$. Thus, using the Lemma \ref{lemma A2}, we obtain
\begin{align*}
\|v\|_{H_7}\leq C_2'\left(\|v\|_{H^2}+\|Av\|_{H^2}\right).
\end{align*}
With this and \eqref{eq 26}, we get that 
\begin{align*}
\|u(\cdot,t)\|_{H^7}\leq \frac{C_2'C_1}{\sqrt{t}}\left(\|u_0\|_{L^2}+\|Au_0\|_{L^2}\right)
=\frac{C_2'C_1}{\sqrt{t}}\|u_0\|_{D(A)}
\leq \frac{C_2}{\sqrt{t}}\|u_0\|_{H^5},
\end{align*}
showing the proposition.
\end{proof}
\begin{proposition}\label{S:H^m->H^m+2}
For every $t \in (0,1]$ and $m \in \{1,2,3,4,5\}$ the map $S(t):H^m\rightarrow H^{m+2}$ is continuous; there exists a constant $C_3>1$ such that $\|u(\cdot,t)\|_{H^{m+2}}\leq C_3\|u_0\|_{H^m},$ for all $u_0 \in X_m.$
\end{proposition}
\begin{proof}
By Propositions \ref{S:L^2->H^2} and \ref{S:H^5->H^7}, the linear maps $S(t):L^2(-1,0)\rightarrow L^2(-1,0)$ and $S(t):D(A)\rightarrow H^7(-1,0)$ are continuous. Moreover, there exists $C_3>1$ such that
\begin{equation}\label{eq 27}
\begin{cases}
\begin{aligned}
&\|S(t)u_0\|_{H^2}\leq \frac{C_3}{\sqrt{t}}\|u_0\|_{L^2},\ \forall\ u_0 \in L^2(-1,0),\\
&\|S(t)u_0\|_{H^7}\leq \frac{C_3}{\sqrt{t}}\|u_0\|_{H^5},\ \forall\ u_0 \in D(A).
\end{aligned}
\end{cases}
\end{equation}
For $m=1,2,3,4$, thanks to the interpolation arguments, it follows that $S(t)\left(\left[X_0,X_5\right]_{\frac{m}{5}}\right)\subset \left[H^2,H^7\right]_{\frac{m}{5}}$, and the maps $S(t):\left[X_0,X_5\right]_{\frac{m}{5}}\rightarrow \left[H^2,H^7\right]_{\frac{m}{5}}$
are continuous with
\begin{align*}
\|S(t)u_0\|_{\left[H^2,H^7\right]_{\frac{m}{5}}}\leq \frac{C_3}{\sqrt{t}}\|u_0\|_{\left[X_0,X_5\right]_{\frac{m}{5}}},\ \forall\ u_0\in  \left[X_0,X_5\right]_{\frac{m}{5}}.
\end{align*}
Since $\left[H^2,H^7\right]_{\frac{m}{5}}=H^{(1-\frac{m}{5})\cdot 2+\frac{m}{2}\cdot 7}=H^{m+2}$ and $\left[X_0,X_5\right]_{\frac{m}{5}}=X_m$, it follows that the maps $S(t):X_m\rightarrow H^{m+2}$ for $m=1,2,3,4$, are continuous with
\begin{align}\label{eq 28}
	\|S(t)u_0\|_{H^{m+2}}\leq \frac{C_3}{\sqrt{t}}\|u_0\|_{H^m},\ \forall\ u_0\in  H^m.
\end{align}
From \eqref{eq 27} and \eqref{eq 28} we obtain the desired.
\end{proof}

It is convenient to remember that for $n \in \mathbb{N}$, we define inductively by 
\begin{align*}
D(A^n)&=\left\{v \in L^2(-1,0);\ v \in D(A^{n-1})\text{ and }Av\in D(A^{n-1})\right\},\ \ A^nv=A^{n-1}(Av).
\end{align*}
So, with this in hand, we have the following result.
\begin{proposition}\label{u in D(A^n)}
Let $u_0 \in L^2(-1,0)$ and $t \in (0,1]$. We have for $u=S(\cdot)u_0$ that:
\begin{enumerate}
	\item [$(i)$] $u(\cdot,t) \in D(A)$ and
	\begin{align*}
	\|Au(\cdot,t)\|_{L^2}\leq \frac{C_4}{t^{\frac{3}{2}}}\|u_0\|_{L^2}
	\end{align*}
for some positive constant $C_4>1$ (which does not depend on $t$).
\item [$(ii)$] $u(\cdot,t)\in D(A^n)$ for every $n\in \mathbb{N}$ and
\begin{align*}
	\|A^nu(\cdot,t)\|_{L^2}\leq \frac{C_5^n}{t^{\frac{3n}{2}}} n^{\frac{3n}{2}}\|u_0\|_{L^2},
\end{align*}
where $C_5>\max\{1,C_4\}$ is a constant which does not depend on $t$.
\item [$(iii)$] $u \in C((0,1],D(A^n))$, for every $n \in \mathbb{N}\cup\{0\}$.
\end{enumerate}
\end{proposition}
\begin{proof}
\textbf{(\textit{i})} Assume $u_0 \in D(A)$. Splitting $[0,t]$ into $[0,t/3]\cup[t/3,2t/3]\cup[2t/3,t]$, using Proposition \ref{S:H^m->H^m+2} and take in mind that $Au(\cdot,t)=-\partial_xu(\cdot,t)-\partial_x^3u(\cdot,t)+\partial_x^5u(\cdot,t),$ we get that 
\begin{align}\label{eq 29}
\|Au(\cdot,t)\|_{L^2}\leq \frac{\tilde{C}_4}{t^{\frac{3}{2}}}\|u_0\|_{L^2},\ \forall\ u_0 \in D(A),
\end{align}
 where $\tilde{C}_4=\frac{3C_3^3}{\left(\sqrt{\frac{1}{3}}\right)^3}$ and $\tilde{C}_4>1$.

Note that, thanks to the inequality \eqref{eq 29}, the  linear operator  $S(t):(D(A),\|\cdot\|_{L^2})\rightarrow (D(A),\|\cdot\|_{D(A)})$ is a bounded linear, and holds that
\begin{align*}
\|S(t)u_0\|_{D(A)}\leq \frac{2\tilde{C}_4}{t^{\frac{3}{2}}}\|u_0\|_{L^2},
\end{align*}
 for $u_0 \in D(A)$. Since $D(A)$ is dense in $L^2(-1,0)$, there exists a bounded linear operator $\Lambda_t:(L^2(-1,0),\|\cdot\|_{L^2})\rightarrow (D(A),\|\cdot\|_{D(A)})$ such that
\begin{align}\label{eq 30}
\Lambda_t\big|_{D(A)}=S(t)\text{ \ \ and \ \ }\|\Lambda_t v\|_{D(A)}\leq \frac{2\tilde{C}_4}{t^{\frac{3}{2}}}\|v\|_{L^2},\ \ \forall\ v \in L^2(-1,0).
\end{align}

Now, given $u_0 \in L^2(-1,0)$, there exists $(u_k)\subset D(A)$ such that $u_k\rightarrow u_0$ in $L^2(-1,0)$. Then from \eqref{eq 30},
\begin{align*}
\|\Lambda_t u_0-S(t)u_0\|_{L^2}\leq \frac{2\tilde{C}_4}{t^{\frac{3}{2}}}\|u_0-u_k\|_{L^2}+\|u_0-u_k\|_{L^2}.
\end{align*}
Making $k\rightarrow \infty$ we obtain $\Lambda_t u_0=S(t)u_0.$ Therefore, $S(t)u_0 \in D(A),$ for  $u_0 \in L^2(-1,0)$, and by \eqref{eq 30}
\begin{align*}
\|Au(\cdot,t)\|_{L^2}=\|AS(t)u_0\|_{L^2}\leq \|S(t)u_0\|_{D(A)}=\|\Lambda_t u_0\|_{D(A)}\leq \frac{C_4}{t^{\frac{3}{2}}}\|u_0\|_{L^2},
\end{align*}
where $C_4=2\tilde{C}_4$, giving the iten (i).

\vspace{0.1cm}

\textbf{(\textit{ii})} Assume $u_0 \in D(A^n)$. From the demigroup theory we have $S(t)u_0\in D(A^n)$ and $A^nu(\cdot,t)=AS(t)A^{n-1}u_0$. Using the item $(i)$ we get
\begin{align}\label{eq 31}
\|A^nu(\cdot,t)\|_{L^2}=\|AS(t)A^{n-1}u_0\|_{L^2}\leq \frac{C_4}{t^{\frac{3}{2}}}\left\|A^{n-1}u_0\right\|_{L^2}.
\end{align}
Splitting $[0,t]$ into $[0,t]=[0,t/n]\cup[t/n,2t/n]\cup\cdots\cup[(n-1)t/n,t]$ and using \eqref{eq 31} several times we obtain
\begin{align*}
\|A^nu(\cdot,t)\|_{L^2}\leq \frac{C_4}{t^{\frac{3}{2}}}\left\|A^{n-1}u(\cdot,(n-1)t/n)\right\|_{L^2}\leq\cdots&\leq \frac{C_4}{t^{\frac{3}{2}}}\frac{C_4}{\left(\frac{n-1}{n}t\right)^\frac{3}{2}}\frac{C_4}{\left(\frac{n-2}{n}t\right)^\frac{3}{2}}\cdots\frac{C_4}{\left(\frac{2t}{n}\right)^\frac{3}{2}}\frac{C_4}{\left(\frac{t}{n}\right)^\frac{3}{2}}\|u_0\|_{L^2}\\
&\leq \left(\frac{C_4}{\left(\frac{t}{n}\right)^\frac{3}{2}}\right)^n\|u_0\|_{L^2},
\end{align*}
thus,  for  $u_0 \in D(A^n)$, holds that
\begin{align*}
\|A^nu(\cdot,t)\|_{L^2}&\leq \frac{C_4^n}{t^{\frac{3n}{2}}} n^\frac{3n}{2}\|u_0\|_{L^2}.
\end{align*}
Now, remark that $S(t):\left(D(A^n),\|\cdot\|_{L^2}\right)\rightarrow\left(D(A^n),\|\cdot\|_{D(A^n)}\right)$
is a bounded linear operator, since, for $u_0 \in D(A^n)$, the previous estimate ensures that
\begin{align*}
\|S(t)u_0\|_{D(A^n)}=\|S(t)u_0\|_{L^2}+\|A^nS(t)u_0\|_{L^2}\leq \left(1+\frac{C_4^n}{t^{\frac{3n}{2}}} n^\frac{3n}{2}\right)\|u_0\|_{L^2}.
\end{align*}
Since $D(A^n)$ is dense in $L^2(-1,0)$, there exists a bounded linear operator  $\Lambda_{t,n}:\left(L^2(-1,0), \|\cdot\|_{L^2}\right)\rightarrow \left(D(A^n), \|\cdot\|_{D(A^n)}\right)$, such that
\begin{align}\label{eq 32}
\Lambda_{t,n}\big|_{D(A^n)}=S(t)\text{ \ \ and \ \ }\|\Lambda_{t,n}v\|_{D(A^n)}\leq \frac{C_5^n}{t^\frac{3n}{2}}n^{\frac{3n}{2}}\|v\|_{L^2},\ \forall v\in L^2(-1,0)
\end{align}
for some constant $C_5>\max\{1,C_4\}$ which does not depend on $t$. 

Given $u_0 \in L^2(-1,0)$, there exists $(u_k)\subset D(A^n)$ such that $u_k\rightarrow u_0$ in $L^2(-1,0)$. Thus, \eqref{eq 32} gives that
\begin{align*}
\|\Lambda_{t,n}u_0-S(t)u_0\|_{L^2}&\leq \|\Lambda_{t,n} u_0-\Lambda_{t,n}u_k\|_{D(A^n)}+\|S(t)u_k-S(t)u_0\|_{L^2}\\
&\leq \frac{C_5^n}{t^\frac{3n}{2}}n^{\frac{3n}{2}}\|u_0-u_k\|_{L^2}+\|u_k-u_0\|_{L^2}.
\end{align*}
Making $k\rightarrow \infty$ we obtain $S(t)u_0=\Lambda_{t,n}u_0$. Therefore, $S(t)u_0 \in D(A),$ for $u_0 \in L^2(-1,0)$, and, due to \eqref{eq 32},
\begin{align*}
\|A^nu(\cdot,t)\|_{L^2}\leq \|S(t)u_0\|_{D(A^n)}=\|\Lambda_{t,n}u_0\|_{D(A^n)}\leq \frac{C_5^n}{t^\frac{3n}{2}}n^{\frac{3n}{2}}\|u_0\|_{L^2}, 
\end{align*}
and item (ii) holds.

\vspace{0.1cm}

\textbf{(\textit{iii})} Let $n \in \mathbb{N}$ and $\varepsilon \in (0,1)$ be. By the item $(ii)$ we have $S(t)u_0 \in D(A^n)$. In particular, $S(\varepsilon)u_0 \in D(A^n)$ so, from the semigroup theory, we obtain $S(\cdot)u_0\in \bigcap_{j=0}^nC^{n-j}([\varepsilon,1];D(A^j))$. 
Taking $j=n$ we get $S(\cdot)u_0\in C([\varepsilon,1];D(A^n))$,  and as $\varepsilon \in (0,1)$ is arbitrary it follows that $S(\cdot)u_0\in C((0,1];D(A^n))$, showing the item (iii), and the proof is finished. 
\end{proof}

The last lemma will be useful to prove the main result of this section.
\begin{lemma}\label{D(A^n) subset H^5n}
For every $n \in \mathbb{N}\cap \{0\}$ we have $D(A^n)\subset H^{5n}(-1,0).$
\end{lemma}
\begin{proof}
The result is obvious for $n=0$ and $n=1$. For $n=2$, given $v \in D(A^2)$ there exists $g \in D(A)\subset H^5$ such that $Av=g$, that is,  $\partial_x^5v=g+\partial_xv+\partial_x^3v \in H^2$ and therefore $v \in H^7$. Thus, $\partial_x^7v=\partial_x^2g+\partial_x^3v+\partial_x^5v\in H^2$, so $v \in H^9$. Deriving, again, we get $\partial_x^9v=\partial_x^4g+\partial_x^5v+\partial_x^7v\in H^1$, and so $v \in H^{10}$. Therefore $D(A^2)\subset H^{10}.$

To conclude the result for any $ n \in \mathbb{N}$ we proceed by induction. Suppose that for some $n\geq 1$ we have $
D(A^n)\subset H^{5n}(-1,0).$ Let $v \in D(A^{n+1})$, 
%and remember that
%\begin{align*}
%D(A^{n+1})=\left\{v \in L^2(-1,0);\ v \in D(A^n)\text{ and }Av \in D(A^n)\right\}.
%\end{align*}
by induction hypothesis, and using the same procedure, we have $v, Av\in H^{5n}$ which implies that there exists $f \in H^{5n}$ with $Av=f$, that is, $\partial_x^5v\in H^{5n-3}$ and so $v \in H^{5n+2}$. By deriving $Av=f$ twice we obtain $v \in H^{5n+4}$. Deriving again twice it follows that $v \in H^{5n+5}$ and therefore $D(A^{n+1})\subset H^{5(n+1)}$, which concludes the proof.
\end{proof}

The previous results ensure directly the following one.
\begin{proposition}%\label{u in C^infty}
For any $u_0 \in L^2(-1,0)$ the solution $u(\cdot,t)=S(t)u_0$ of \eqref{kawahara control system} with $h_1=h_2=0$ satisfies $u \in C^\infty([-1,0]\times (0,1]).$
\end{proposition}
\begin{proof}
Consider $n \in \mathbb{N}$, $t \in (0,1]$ and $u_0 \in L^2(-1,0)$. From Proposition \ref{u in D(A^n)} we have $u(\cdot,t)\in D(A^n)$. Using Lemma \ref{D(A^n) subset H^5n} we obtain $u(\cdot,t)\in H^{5n}$. The Sobolev embedding $H^{5n}\hookrightarrow C^{5n-1}([-1,0])$ provides $u(\cdot,t)\in C^{5n-1}([-1,0])$. Since $n \in \mathbb{N}$ is arbitrary, it follows that $u(\cdot,t)\in C^\infty([-1,0]).$ By the other hand, given $\varepsilon>0$ and $n \in \mathbb{N}$, Proposition \ref{u in D(A^n)} yields that $u(\cdot,\varepsilon) \in D(A^{n+1})$ so, the semigroup theory, follows that  $u \in \bigcap_{j=0}^{n+1}C^{n+1-j}\left([\varepsilon,1],D(A^j)\right).$  In particular, taking $j=1$, we have that $u \in C^n([\varepsilon,1],D(A)).$ Using the Lemma \ref{lemma A2} and the embedding $H^5\hookrightarrow C^4([-1,0])$ we obtain that $\partial_t^iu(x,\cdot)$ is continuous at $t_0$ and, as $t_0\in [\varepsilon,1]$ is arbitrary we conclude that $\partial_t^iu(x,\cdot)$ is continuous for $i=0,1,...,n$. Since $n \in \mathbb{N}$ and $x \in [-1,0]$ are arbitrary we get $u(x,\cdot)\in C^\infty([\varepsilon,1]$. Furthermore,  $u(x,\cdot)\in C^\infty((0,1]),$ and the result holds.
\end{proof}

We are now in a position to prove a smooth property for the solutions of \eqref{kawahara control system}. With these auxiliary results in hand, the main result of this section is the following one.
\begin{proposition}\label{smoothing efect}
Let $u_0 \in L^2(-1,0)$ and $h_1(t)=h_2(t)=0$ for $t \in [0,1]$. Then the corresponding solution $u$ of \eqref{kawahara control system} satisfies $u \in G^{\frac{1}{2},\frac{5}{2}}\left([-1,0]\times[\varepsilon,1]\right)$, for all $\varepsilon \in (0,1)$, that is, we can find positive constants $M,R_1,R_2$ such that
\begin{align*}
\left|\partial_x^p\partial_t^nu(x,t)\right|\leq \frac{M}{t^\frac{5n+p+5}{2}}\frac{p!^{\frac{1}{2}}}{R_1^p}\frac{n!^{\frac{5}{2}}}{R_2^n}.
\end{align*}
\end{proposition}
\begin{proof}
Consider $p \in \mathbb{N}\cup\{0\}$ and choose $n \in \mathbb{N}$ such that $5n-5\leq p\leq 5n-1$. Then, the Sobolev embedding, Lemma \ref{lemma A3} and Proposition \ref{u in D(A^n)} ensures that
\begin{equation}\label{eq 38}
\begin{split}
\left\|\partial_x^pu(\cdot,t)\right\|_\infty
%\leq C_6\|\partial_x^pu(\cdot,t)\|_{H^1}\leq C_6\|u(\cdot,t)\|_{H^{p+1}}
%&\leq C_6\|u(\cdot,t)\|_{H^{5n}}\\
&\leq C_6K_3^n\sum_{i=0}^n\left\|P^iu(\cdot,t)\right\|_{L^2}\\%=C_6K_3^n\sum_{i=0}^n\left\|A^iu(\cdot,t)\right\|_{L^2}\\
&=C_6K_3^n\left(\|u(\cdot,t)\|_{L^2}+\sum_{i=1}^n\left\|A^iu(\cdot,t)\right\|_{L^2}\right)\\
&\leq C_6K_3^n\left(1+\sum_{i=1}^n\frac{C_5^i}{t^\frac{3i}{2}}i^{\frac{3i}{2}}\right)\|u_0\|_{L^2}\\
&\leq C_6K_3^n(n+1)\frac{C_5^n}{t^\frac{3n}{2}}n^{\frac{3n}{2}}\|u_0\|_{L^2}\\
&=C_6K_3^n(n+1)\frac{C_5^n}{t^\frac{5n}{2}}\frac{(5n)^\frac{5n}{2}}{5^\frac{5n}{2}}\|u_0\|_{L^2}.
\end{split}
\end{equation}
Using the Stirling's formula
%\begin{align*}
%(5n)!\sim \left(\frac{5n}{e}\right)^{5n}\sqrt{2\pi 5n}=\frac{(5n)^{5n}}{e^{5n}}(2\pi)^\frac{1}{2}(5n)^\frac{1}{2}
%\end{align*}
we get
\begin{align*}
(5n)^{5n}\sim\frac{e^{5n}(5n)!}{(2\pi)^\frac{1}{2}(5n)^\frac{1}{2}}\iff (5n)^{\frac{5n}{2}}\sim\frac{e^{\frac{5n}{2}}(5n)!^\frac{1}{2}}{(2\pi)^\frac{1}{4}(5n)^\frac{1}{4}}\iff \frac{(5n)^{\frac{5n}{2}}}{5^\frac{5n}{2}}\sim \frac{1}{(2\pi)^\frac{1}{4}(5n)^\frac{1}{4}}\left(\frac{e}{5}\right)^\frac{5n}{2}(5n)!^\frac{1}{2},
\end{align*}
which implies that
\begin{align*}
\frac{(5n)^{\frac{5n}{2}}}{5^\frac{5n}{2}}\leq C_7 \frac{1}{(2\pi)^\frac{1}{4}(5n)^\frac{1}{4}}\left(\frac{e}{5}\right)^\frac{5n}{2}(5n)!^\frac{1}{2}\leq C_7\left(\frac{e}{5}\right)^\frac{5n}{2}(5n)!^\frac{1}{2}
\end{align*}
for some constant $C_7>0$. Since $(n+1)\leq e^{\frac{5n}{2}},$ for $n \in \mathbb{N},$ the inequality \eqref{eq 38} gives us
\begin{align*}
\left\|\partial_x^pu(\cdot,t)\right\|_\infty&\leq C_6K_3^ne^{\frac{5n}{2}}\frac{C_5^n}{t^\frac{5n}{2}}C_7\left(\frac{e}{5}\right)^\frac{5n}{2}(5n)!^\frac{1}{2}\|u_0\|_{L^2}=C_6C_7\frac{C_5^n}{t^\frac{5n}{2}}K_3^n\left(\frac{e^2}{5}\right)^\frac{5n}{2}(5n)!^\frac{1}{2}\|u_0\|_{L^2}.
\end{align*}
Remember that $5n-5\leq p\leq 5n-1$ which allow us write $p=5n-r$ with $r\in\{1,2,3,4,5\}$, that is, $5n=p+r,$ with $r \in \{1,2,3,4,5\}.$
Using \eqref{inequality} it follows that
\begin{align*}
\left\|\partial_x^pu(\cdot,t)\right\|_\infty&\leq C_6C_7\frac{C_5^\frac{p+r}{5}}{t^\frac{p+r}{2}}K_3^\frac{p+r}{5}\left(\frac{e^2}{5}\right)^\frac{p+r}{2}(p+r)!^\frac{1}{2}\|u_0\|_{L^2}\\
&\leq C_6C_7C_5^\frac{r}{5}K_3^\frac{r}{5}\left(\frac{e^2}{5}\right)^\frac{r}{2}\frac{C_5^\frac{p}{5}}{t^\frac{p+r}{2}}K_3^\frac{p}{5}\left(\frac{e^2}{5}\right)^\frac{p}{2}\left(2^p2^rp!r!\right)^\frac{1}{2}\|u_0\|_{L^2}\\
&=C_6C_7C_5^\frac{r}{5}K_3^\frac{r}{5}\left(\frac{e^2}{5}\right)^\frac{r}{2}2^{\frac{r}{2}}r!^\frac{1}{2}\|u_0\|_{L^2}\frac{1}{t^\frac{p+r}{2}}\left(\frac{C_5^\frac{1}{5}K_3^\frac{1}{5}e2^\frac{1}{2}}{5^\frac{1}{2}}\right)^pp!^\frac{1}{2}.
\end{align*}
Consequently
\begin{align*}
\left\|\partial_x^pu(\cdot,t)\right\|_\infty&\leq \frac{C_8}{t^\frac{p+r}{2}}\left(\frac{C_5^\frac{1}{5}K_0^\frac{1}{5}e2^\frac{1}{2}}{5^\frac{1}{2}}\right)^pp!^\frac{1}{2}=\frac{C_8}{t^\frac{p+r}{2}}\frac{p!^\frac{1}{2}}{\left(5^\frac{1}{2}C_5^{-\frac{1}{5}}K_0^{-\frac{1}{5}}e^{-1}2^{-\frac{1}{2}}\right)^p},
\end{align*}
with
$$
C_8:=C_6C_7C_5^\frac{r}{5}K_3^\frac{r}{5}\left(\frac{e^2}{5}\right)^\frac{r}{2}2^{\frac{r}{2}}r!^\frac{1}{2}\|u_0\|_{L^2}\text{ \ \ and \ \ }K_0=K_3+1.
$$
Defining $R=5^\frac{1}{2}C_5^{-\frac{1}{5}}K_0^{-\frac{1}{5}}e^{-1}2^{-\frac{1}{2}}$, follows that $R \in (0,1)$ and
\begin{align*}
\left\|\partial_x^pu(\cdot,t)\right\|_\infty&\leq \frac{C_8}{t^\frac{p+r}{2}}\frac{p!^\frac{1}{2}}{R^p}.
\end{align*}
Here, $p\geq 0$, $r\in \{1,2,3,4,5\}$ and $t \in (0,1]$. Observe that, as $t\in(0,1]$ and $0\leq r\leq 5$ holds that 
\begin{align}\label{eq 39}
	\left\|\partial_x^pu(\cdot,t)\right\|_\infty&\leq \frac{C_8}{t^\frac{p+5}{2}}\frac{p!^\frac{1}{2}}{R^p},\ \forall\ p\geq 0,\ \forall t \in (0,1].
\end{align}

Finally, for every $n,p\geq 0$ from \eqref{d^n_tu} we have
\begin{align*}
\partial_t^n\partial_x^pu&=\partial_x^p\partial_t^nu=\partial_x^p\left(-(-1)^{n-1}P^nu\right)=\partial_x^p\left((-1)^{n}P^nu\right)=(-1)^{n}P^n\partial_x^pu.
\end{align*}
Using Newton's Binomial theorem we have that
\begin{align*}
P^n\partial_x^pu
%&=(-\partial_x^5+\partial_x^3+\partial_x)^n\partial_x^pu\\&=\sum_{q=0}^n\binom{n}{q}\left(-\partial_x^5+\partial_x^3\right)^q\partial_x^{n-q}\partial_x^pu\\
&=\sum_{q=0}^n\binom{n}{q}\sum_{j=0}^q\binom{q}{j}(-\partial_x^5)^j(\partial_x^3)^{q-j}\partial_x^{n-q}\partial_x^pu\\
&=\sum_{q=0}^n\binom{n}{q}\sum_{j=0}^q\binom{q}{j}(-1)^j\partial_x^{5j}\partial_x^{3q-3j}\partial_x^{n-q}\partial_x^pu.
\end{align*}
Thus, from \eqref{eq 39}, for $(x,t)\in [-1,0]\times(0,1]$,
\begin{align*}
\left|\partial_t^n\partial_x^pu(x,t)\right|&\leq \sum_{q=0}^n\binom{n}{q}\sum_{j=0}^q\binom{q}{j}\left|\partial_x^{n+2q+2j+p}u(x,t)\right|\\
&\leq \sum_{q=0}^n\binom{n}{q}\sum_{j=0}^q\binom{q}{j}\frac{C_8}{t^\frac{n+2q+2j+p+5}{2}}\frac{(n+2q+2j+p)!^\frac{1}{2}}{R^{n+2q+2j+p}}.
\end{align*}
Once that $t \in (0,1]$, $R<1$ and $n+2q+2j+p\leq 5n+p$ it follows that
\begin{align*}
\left|\partial_t^n\partial_x^pu(x,t)\right|&\leq \frac{C_8}{t^\frac{5n+p+5}{2}}\frac{(5n+p)!^\frac{1}{2}}{R^{5n+p}}\sum_{q=0}^n\binom{n}{q}\sum_{j=0}^q\binom{q}{j}.
\end{align*}
Noting that
\begin{align*}
2^q=(1+1)^q=\sum_{j=0}^q\binom{q}{j}1^j\cdot 1^{q-j}=\sum_{j=0}^q\binom{q}{j}
\end{align*}
and
\begin{align*}
\sum_{q=0}^n\binom{n}{q}\sum_{j=0}^q\binom{q}{j}=\sum_{q=0}^n\binom{n}{q}2^q=\sum_{q=0}^n\binom{n}{q}2^q\cdot 1^{n-q}=(2+1)^n=3^n,
\end{align*}
using \eqref{inequality} one more time, yields
\begin{align*}
	\left|\partial_t^n\partial_x^pu(x,t)\right|&\leq \frac{C_8}{t^\frac{5n+p+5}{2}}\frac{3^n\cdot 2^\frac{5n}{2}2^\frac{2}{p}(5n)!^\frac{1}{2}p!^\frac{1}{2}}{R^{5n}R^p}.
\end{align*}
The Stirling's formula gives us
\begin{align}\label{eq 40}
(5n)!\sim \left(\frac{5n}{e}\right)^{5n}\sqrt{2\pi 5n}=(10\pi)^\frac{1}{2}\left(\frac{5n}{e}\right)^{5n}n^\frac{1}{2}.
\end{align}
Moreover, we also have $n!\sim\left(\frac{n}{e}\right)^n\sqrt{2\pi n}=(2\pi)^\frac{1}{2}\left(\frac{n}{e}\right)^nn^\frac{1}{2},$ then
\begin{align*}
\frac{5^{5n}n!^5}{(2\pi)^\frac{5}{2}n^\frac{5}{2}}\sim \left(\frac{5n}{e}\right)^{5n}.
\end{align*}
This previous relation together with \eqref{eq 40} leads us
\begin{align*}
(5n)!\sim (10\pi)^\frac{1}{2}\frac{5^{5n}n!^5}{(2\pi)^\frac{5}{2}n^\frac{5}{2}}n^\frac{1}{2}=\frac{(10\pi)^\frac{1}{2}}{(2\pi)^\frac{5}{2}}\frac{5^{5n}n!^5}{n^2} \sim \frac{(10\pi)^\frac{1}{4}}{(2\pi)^\frac{5}{4}}\frac{5^\frac{5n}{2}n!^\frac{5}{2}}{n}.
\end{align*}
Thus, there exists a constant $C_9>0$ such that $(5n)!^\frac{1}{2}\leq C_95^\frac{5n}{2}n!^\frac{5}{2}.$ Hence
\begin{align*}
\left|\partial_t^n\partial_x^pu(x,t)\right|&\leq \frac{C_8C_9}{t^\frac{5n+p+5}{2}}\left(\frac{3\cdot 2^\frac{5}{2}\cdot 5^\frac{5}{2}}{R^5}\right)^nn!^\frac{5}{2}\left(\frac{2^\frac{1}{2}}{R}\right)^pp!^\frac{1}{2}=\frac{C_8C_9}{t^\frac{5n+p+5}{2}}\frac{n!^\frac{5}{2}}{\left(3^{-1}\cdot 10^{-\frac{5}{2}}\cdot R^5\right)^n}\frac{p!^\frac{1}{2}}{\left(2^{-\frac{1}{2}}R\right)^p}
\end{align*}
and the result is achieved with $M=C_8C_9,$ $R_1=2^{-\frac{1}{2}}R$ and $R_2=3^{-1}\cdot 10^{-\frac{5}{2}}\cdot R^5.$
\end{proof}
\subsection{Null controllability results} Let us now prove the null controllability result. Precisely, employing two flat output controls, the solution of \eqref{kawahara control system} satisfies $u(\cdot,T)=0$.

\begin{proof}[Proof of Theorem \ref{null controllability}] Consider $u_0 \in L^2(-1,0)$ and denote by $\overline{u}$ the solution of \eqref{kawahara control system} for $h_1=h_2=0$. From Proposition \ref{smoothing efect} we have, for $\varepsilon\in(0,T)$ , that $\overline{u} \in G^{\frac{1}{2},\frac{5}{2}}([-1,0]\times[\varepsilon,T]).$  In particular, $\partial_x^3\overline{u}(0,t),\partial_x^4\overline{u}(0,t)\in G^\frac{5}{2}([\varepsilon,T])$ for any $\varepsilon \in (0,T)$. Choose $\tau \in (0,T)$ and define
	\begin{align*}
y(t)=\phi_s\left(\frac{t-\tau}{T-\tau}\right)\partial_x^3\overline{u}(0,t)\text{ \ \ and \ \ }z(t)=\phi_s\left(\frac{t-\tau}{T-\tau}\right)\partial_x^4\overline{u}(0,t),
	\end{align*}
where $\phi_s$ is the step function given by 
\begin{align*}
\phi_s(r)=\left\{
\begin{array}{ll}
1,&\text{ if }r\leq 0,\\
0,&\text{ if }r\geq 1,\\
\frac{e^{-\frac{K}{(1-r)^\sigma}}}{e^{-\frac{K}{r^\sigma}}+e^{-\frac{K}{(1-r)^\sigma}}},&\text{ if }r \in (0,1),
\end{array}
\right.
\end{align*}
with $K>0$ and $\sigma:=(s-1)^{-1}$. As $\phi_s$ is Gevrey of order $s$ (see, for example, \cite{1D Schrödinger by flatness}) and $s\geq \frac{5}{2}$ we infer that $y,z \in G^s\left([\varepsilon,T]\right),\ \forall\ \varepsilon \in (0,T).$ Then, defining $u:[-1,0]\times(0,T]\rightarrow\mathbb{R}$ by
\begin{align*}
u(x,t)&=\left\{
\begin{array}{ll}
u_0(x),& \text{ if }t=0,\\
\sum_{i\geq 0}f_i(x)y^{(i)}(t)+\sum_{i\geq 0}g_i(x)z^{(i)}(t),&\text{ if } t \in(0,T],
\end{array}
\right.
\end{align*}
the Proposition \ref{flatness well-posedness} gives us that $u$ satisfies \eqref{flatness system} with $u \in G^{\frac{s}{5},s}\left([-1,0]\times [\varepsilon,T]\right)$ for all $\varepsilon \in (0,T)$. In particular
$$
\partial_x^mu(0,t)=0,\ m=0,1,2,\ \  \partial_x^3u(0,t)=y(t) \text{ \ \ and \ \ } \partial_x^4u(0,t)=z(t).
$$
Furthermore, by construction
\begin{align*}
\partial_x^m\overline{u}(0,t)=0,\ m=0,1,2
\end{align*}
and, for $t \in (0,\tau)$
\begin{align*}
y(t)=\underbrace{\phi_s\left(\frac{t-\tau}{T-\tau}\right)}_{=1}\partial_x^3\overline{u}(0,t)=\partial_x^3\overline{u}(0,t),
\end{align*}
and
\begin{align*}
z(t)=\underbrace{\phi_s\left(\frac{t-\tau}{T-\tau}\right)}_{=1}\partial_x^4\overline{u}(0,t)=\partial_x^4\overline{u}(0,t).
\end{align*}
Therefore $\partial_x^mu(0,t)=\partial_x^m\overline{u}(0,t)$, for $m=0,1,2,3,4$ and $t \in (0,\tau)$. Thanks to the Holmgren theorem, we conclude that $u(x,t)=\overline{u}(x,t)$, for all $(x,t)\in [-1,0]\times (0,\tau)$. Hence, $u \in C([0,T],L^2(-1,0))$ and it solves the system \eqref{kawahara control system} with 
\begin{align*}
h_1(t)=\sum_{i\geq 0}f_i(-1)y^{(i)}(t)+\sum_{i\geq 0}g_i(-1)z^{(i)}(t)
\end{align*}
and
\begin{align*}
h_2(t)=\sum_{i\geq 0}f_{ix}(-1)y^{(i)}(t)+\sum_{i\geq 0}g_{ix}(-1)z^{(i)}(t).
\end{align*}
Observe that $h_1,h_2 \in G^s([0,T])$ and $h_1(t)=h_2(t)=0$ for $0<t<\tau$ since
\begin{align*}
\begin{cases}
\begin{array}{l}
h_1(t)=u(-1,t)=\overline{u}(-1,t)=0,\\
h_2(t)=u_x(-1,t)=\overline{u}_x(-1,t)=0,
\end{array}\ \forall\ t \in (0,\tau).
\end{cases}
\end{align*}
Finally, once we have $\supp y^{(i)}\subset \supp y\subset (-\infty,T)$ and $\supp z^{(i)}\subset \supp z\subset (-\infty,T)$, follows that $y^{(i)}(T)=0$ and $z^{(i)}(T)=0$, for every $i\geq 0$, so that $u(\cdot,T)=0$, and the first main result is showed. 
\end{proof}

\section{A class of reachable functions}\label{sec3}

In this section, we will establish a class of sets that can be reachable from $0$ by the system \eqref{kawahara control system}.  Our goal is to prove that, given $u_1 \in \mathcal{R}_R$ (see the definition in \eqref{RR}), one can find control inputs $h_1$ and $h_2$ for which the solution of \eqref{kawahara control system}, with $u_0=0$, satisfies $u(x,T)=u_1(x)$. 

\subsection{Auxiliary results} To prove what we mentioned before we need auxiliary results. The first establishes the flatness property for the limit case $s=5$.
\begin{proposition}\label{flatness for s=5}
Let $R>1$ and $y,z \in G^5([0,T])$ with
\begin{align}\label{eq 63a}
|y^{(j)}(t)|,|z^{(j)}(t)|\leq M\frac{(5j)!}{R^{5j}},\ \forall j\geq 0,\ \forall t \in [0,T].
\end{align}
Then, defining $u(x,t)$ as in \eqref{flatness solution} we have $u \in G^{1,5}([-1,0]\times[0,T])$ and it solves \eqref{flatness system}.
\end{proposition}
\begin{proof}
Let $m,k\geq 0$. For $j\geq k$ we have, from \eqref{eq 63a} and Lemma \ref{estimetes to f_j and g_j}, that 
\begin{align*}
\sum_{j=0}^\infty\left|f_{j-k}(x)y^{(j+m)}(t)\right|\leq \sum_{j\geq k}2^{j-k}\frac{|x|^{5(j-k)+1}}{\left[5(j-k)+1\right]!}M\frac{[5(j+m)]!}{R^{5(j+m)}}.
\end{align*}
Pick $l=5(j-k)$ and $N=5(k+m)$ so $l+N=5(j+m)$, gives that
\begin{align*}
\sum_{j=0}^\infty\left|f_{j-k}(x)y^{(j+m)}(t)\right|\leq M\sum_{l\geq 0}2^\frac{l}{5}\frac{1}{(l+1)!}\frac{(l+N)!}{R^{l+N}}.
\end{align*}
If $N\leq 1$ then
\begin{align*}
\sum_{j=0}^\infty\left|f_{j-k}(x)y^{(j+m)}(t)\right|\leq \frac{M}{R^N}\sum_{l\geq 0}\left(\frac{1}{R\cdot 2^{-\frac{1}{5}}}\right)^l<\infty.
\end{align*}
Assume $N>1$, note that $(l+N)!=(l+N)(l+N-1)\cdots(l+2)(l+1)!$, so that
\begin{align*}
\sum_{j=0}^\infty\left|f_{j-k}(x)y^{(j+m)}(t)\right|&\leq M\sum_{l\geq 0}2^{\frac{l}{5}}\frac{(l+N)(l+N-1)\cdots (l+2)}{R^{l+N}}\\
&\leq M\sum_{q\geq 0}\sum_{qN\leq l<(q+1)N}\frac{2^\frac{l+N}{5}(l+N)^{N-1}}{R^{l+N}}\\
&=M\sum_{q\geq 0}\sum_{qN\leq l<(q+1)N}\frac{[(q+2)N]^{N-1}}{\left(2^{-\frac{1}{5}}R\right)^{l+N}}.
\end{align*}
For $l>qN$ we have $l+N\geq (q+1)N$ and, since $2^{-\frac{1}{5}}R>1$, follows that
$$
\frac{1}{\left(2^{-\frac{1}{5}}R\right)^{l+N}}\leq \frac{1}{\left(2^{-\frac{1}{5}}R\right)^{(q+1)N}}.
$$
Thus, thanks to the relation  \eqref{eq 32a}, the following estimate
\begin{align*}
\sum_{j=0}^\infty\left|\partial_t^mP^k\left(f_j(x)y^{(j)}(t)\right)\right|&\leq M\sum_{q\geq 0}\sum_{qN\leq l<(q+1)N}\frac{(q+2)^{N-1}N^{N-1}}{\left(2^{-\frac{1}{5}}R\right)^{(q+1)N}}\\
%&\leq M\sum_{q\geq 0}N\frac{(q+2)^{N-1}N^{N-1}}{\tilde{R}^{(q+1)N}}\\
&\leq MN^N\sum_{q\geq 0}\left(\frac{q+2}{\tilde{R}^{q+1}}\right)^N,
\end{align*}
is verified with $\tilde{R}=2^{-\frac{1}{5}}R$. Pick any $\sigma \in (0,1)$ and define $f_\sigma(x)=\frac{x+2}{\left(\tilde{R}^{1-\sigma}\right)^{x+1}}.$ Note that $\lim_{x\rightarrow\infty}\left(\tilde{R}^{1-\sigma}\right)^{x+1}=\infty$, since $\tilde{R}>1$. The L'Hospital rule ensures that 
\begin{align*}
\lim_{x\rightarrow\infty}f_\sigma(x)=\lim_{x\rightarrow\infty}\frac{1}{\ln(\tilde{R}^{1-\sigma})(\tilde{R}^{1-\sigma})^{x+1}}=0,
\end{align*} 
and so, 
\begin{align*}
\frac{q+2}{\left(\tilde{R}^{1-\sigma}\right)^{q+1}}\rightarrow0,\text{ \ as \ }q\rightarrow\infty.
\end{align*}
Defining $a:=\sup_{q\geq 0}\frac{q+2}{\left(\tilde{R}^{1-\sigma}\right)^{q+1}}$ holds that 
%\begin{align*}
%\left(\frac{q+2}{\tilde{R}^{q+1}}\right)^N\leq \frac{a^N}{\tilde{R}^{N\sigma(q+1)}}=\frac{a^N}{\tilde{R}^{N\sigma}}\cdot \frac{1}{\tilde{R}^{N\sigma q}},
%\end{align*}
%hence
\begin{align*}
\sum_{q\geq 0}\left(\frac{q+2}{\tilde{R}^{q+1}}\right)^N\leq \frac{a^N}{\tilde{R}^{N\sigma}}\sum_{q\geq 0}\frac{1}{\tilde{R}^{N\sigma q}}=\frac{a^N}{\tilde{R}^{N\sigma}}\left(\frac{1}{1-\frac{1}{\tilde{R}^{N\sigma}}}\right).
\end{align*}
Once we have the following convergence
\begin{align*}
\alpha_N:=\frac{1}{1-\frac{1}{\tilde{R}^{N\sigma}}}\rightarrow 1, \text{ \ as \ }N\rightarrow\infty,
\end{align*}
we can define $\displaystyle\tilde{M}:=\sup_{N>1}\alpha_N$ to get
\begin{align*}
\sum_{q\geq 0}\left(\frac{q+2}{\tilde{R}^{q+1}}\right)^N\leq\tilde{M}\frac{a^N}{\tilde{R}^{N\sigma}}.
\end{align*}
Consequently
\begin{align*}
\sum_{j=0}^\infty\left|\partial_t^mP^k\left(f_j(x)y^{(j)}(t)\right)\right|&\leq MN^N\tilde{M}\frac{a^N}{\tilde{R}^{N\sigma}},
\end{align*}
and using Stirling's formula we get that
\begin{align*}
\frac{a^NN^N}{\tilde{R}^{N\sigma}}\sim \frac{1}{\sqrt{2\pi}}\frac{a^Ne^N}{\tilde{R}^{N\sigma}}\frac{N!}{N^\frac{1}{2}},
\end{align*}
which ensures the following estimate
\begin{align*}
\sum_{j=0}^\infty\left|\partial_t^mP^k\left(f_j(x)y^{(j)}(t)\right)\right|&\leq M'\left(\frac{ae}{\tilde{R}^\sigma}\right)^N\frac{N!}{N^\frac{1}{2}},
\end{align*}
for some constant $M'>0$. Moreover, noting that $N!=(5k+5m)!\leq 2^{5k}2^{5m}(5k)!(5m)!$ and using Stirling's formula again, namely, $(5m)!\sim 5^{5m+\frac{1}{2}}\left(\sqrt{2\pi m}\right)^{-4}m!^5$, follows that 
\begin{align*}
\sum_{j=0}^\infty\left|\partial_t^mP^k\left(f_j(x)y^{(j)}(t)\right)\right|
%&\leq M''\left(\frac{ae}{\tilde{R}^\sigma}\right)^{5k+5m}2^{5k+5m}(5k)!\left(\frac{\sqrt{5}\cdot 5^{5m}}{(2\pi m)^2}\right)m!^5\frac{1}{(5k+5m)^\frac{1}{2}}\\
&\leq \sqrt{5}M''\left(\frac{2ae}{\tilde{R}^\sigma}\right)^{5k}(5k)!\left(\frac{10ae}{\tilde{R}^\sigma}\right)^{5m}m!^5\frac{1}{(k+1)^\frac{1}{2}}.
\end{align*}

Now, define $R_1=(2ae)^{-1}\tilde{R}^\sigma$, $R_2=\left[(10ae)^{-1}\tilde{R}\right]^5$ and $M'''=\sqrt{5}M''$, it follows that
\begin{align*}
\sum_{j=0}^\infty\left|\partial_t^mP^k\left(f_j(x)y^{(j)}(t)\right)\right|&\leq M'''\frac{(5k)!}{R_1^{5k}}\frac{m!^5}{R_2^m}\frac{1}{(k+1)^\frac{1}{2}}.
\end{align*}
Observe that we can assume $R_1<1$. Let $K_3>0$ as in Lemma \ref{lemma A3} for $p=\infty$. Then, the previous inequality yields that
 \begin{align*}
\sum_{j=0}^\infty\left\|\partial_t^m\left(f_jy^{(j)}(t)\right)\right\|_{5i,\infty}
%&\leq K_3^i\sum_{j=0}^\infty\sum_{k=0}^i\left\|P^k\partial_t^m\left(f_jy^{(j)}(t)\right)\right\|_\infty\\
%&\leq K_3^i\sum_{k=0}^i\sum_{j=0}^\infty\left\|\partial_t^mP^k\left(f_jy^{(j)}(t)\right)\right\|_\infty\\
%&\leq K_3^i\sum_{k=0}^iM'''\frac{(5k)!}{R_1^{5k}}\frac{m!^5}{R_2^m}\frac{1}{(k+1)^\frac{1}{2}}\\
&\leq M'''K_3^i\frac{m!^5}{R_2^m}\frac{(5i)!}{R_1^{5i}}\sum_{k=0}^i\frac{1}{(k+1)^\frac{1}{2}},
\end{align*}
for all $i\geq 0$. Given $m,n\geq 0$ consider $i\geq 0$ such that $n \in \{5i-r,\ r=0,1,2,3,4\}$. Thus, 
\begin{align*}
\sum_{j=0}^\infty\left|\partial_x^n\partial_t^m\left(f_j(x)y^{(j)}(t)\right)\right|
%&\leq \sum_{j=0}^\infty \left\|\partial_x^n\partial_t^m\left(f_jy^{(j)}(t)\right)\right\|_\infty\\
%&\leq \sum_{j=0}^\infty \left\|\partial_t^m\left(f_jy^{(j)}(t)\right)\right\|_{5i,\infty}\\
&\leq M'''K_3^i\frac{m!^5}{R_2^m}\frac{(5i)!}{R_1^{5i}}\sum_{k=0}^i\frac{1}{(k+1)^\frac{1}{2}},
\end{align*}
for $(x,t)\in [-1,0]\times [0,T]$. Analogously one can see that
\begin{align*}
\sum_{j=0}^\infty\left|\partial_x^n\partial_t^m\left(g_j(x)z^{(j)}(t)\right)\right|&\leq M'''K_3^i\frac{m!^5}{R_2^m}\frac{(5i)!}{R_1^{5i}}\sum_{k=0}^i\frac{1}{(k+1)^\frac{1}{2}}.
\end{align*}
Therefore these series are uniformly convergent on $[-1,0]\times [0,T]$, for all $m,n\geq 0$ so that, the function  $u$ defined by \eqref{flatness solution} satisfies $u \in C^\infty([-1,0]\times [0,T])$. Furthermore,
\begin{align*}
%\left|\partial_x^n\partial_t^mu(x,t)\right|&\leq 
\sum_{j=0}^\infty\left|\partial_x^n\partial_t^m\left(f_j(x)y^{(j)}(t)\right)\right|+\sum_{j=0}^\infty\left|\partial_x^n\partial_t^m\left(g_j(x)z^{(j)}(t)\right)\right|
&\leq M'''K_3^i\frac{m!^5}{R_2^m}\frac{(5i)!}{R_1^{5i}}\sum_{k=0}^i\frac{1}{(k+1)^\frac{1}{2}}.
\end{align*}
Since $n=5i-r$ with $r \in \{0,1,2,3,4\}$ we have
\begin{align*}
\frac{K_3^i}{R_1^{5i}}(5i)!\leq \frac{K_3^{\frac{r}{5}}\cdot 2^r\cdot r!}{R_1^r}\cdot \frac{n!}{\left(K_3^{-\frac{1}{5}}\cdot 2^{-1}\cdot R_1\right)^n}.
\end{align*}
Defining
\begin{align*}
\hat{M}=2M'''\left(\max_{0\leq r\leq 4}\frac{K_3^{\frac{r}{5}}\cdot 2^r\cdot r!}{R_1^r}\right)\sum_{k=0}^i\frac{1}{(k+1)^\frac{1}{2}}\text{ \ \ and \ \ }R_1'=K_3^{-\frac{1}{5}}\cdot 2^{-1}\cdot R_1
\end{align*}
it follows that
\begin{align*}
\left|\partial_x^n\partial_t^mu(x,t)\right|&\leq \hat{M}\frac{n!}{(R_1')^n}\frac{m!^5}{R_2^m}\ \ \forall\ n,m\geq 0,\ \ \forall (x,t)\in [-1,0]\times [0,T],
\end{align*}
which concludes the proof.
\end{proof}

The next result is a particular case of \cite[Proposition 3.6]{Reachable states for heat equation} with $a_0=1$ and $a_p=[5p(5p-1)(5p-2)(5p-3)(5p-4)]^{-1}$, for $p\geq 1$.

\begin{proposition}\label{prop 3.2}
Let $(d_q)_{q\geq 0}$ be a sequence of real numbers satisfying $|d_q|\leq CH^q(5q)!$, for all $q\geq 0$ and for some constants $H>0$ and $C>0$. Then, for each $\tilde{H}>e^{e^{-1}}H$, one can find a function $f \in C^\infty(\mathbb{R})$ such that $f^{(q)}(0)=d_q$, for all $q\geq 0$, and
$$
\begin{aligned}
&\left|f^{(q)}(x)\right|\leq C\tilde{H}^q(5q)!,\quad \forall q\geq 0,\ \forall x \in \mathbb{R}.
\end{aligned}
$$
\end{proposition}

The next lemma is a consequence of the theory of analytic functions.
\begin{lemma}\label{lemma 3.1}
Let $\psi \in G^1([-1,0])$ be such that
$
\partial_x^jP^n\psi(0)=0,\ \forall n \geq 0,\ j=0,1,2,3,4.
$
Then $\psi\equiv0$.
\end{lemma}
\begin{proof}
First, remember (from the proof of Proposition \ref{smoothing efect}) that 
\begin{align*}
	P^n&=\sum_{q=0}^n\binom{n}{q}\sum_{k=0}^q\binom{q}{k}(-1)^k\partial_x^{n+2q+2k}\ \ \ \forall n\geq 0
\end{align*}
so
\begin{align*}
\partial_x^jP^n=\sum_{q=0}^n\binom{n}{q}\sum_{k=0}^q\binom{q}{k}(-1)^k\partial_x^{n+2q+2k+j}\ \ \ \forall n,j\geq 0.
\end{align*}
We claim that for any $n\geq 0$
\begin{align}\label{eq 63}
\partial_x^j\psi(0)=0\ \ \forall j \in \{0,1,...,5n+4\}.
\end{align}
To prove this, we use induction on $n$. For $n=0$ this immediately follows from the hypothesis, since
\begin{align}\label{eq 64}
\partial_x^j\psi(0)=\partial_x^jP^0\psi(0)=0,\ \ \ j=0,1,2,3,4.
\end{align}
Let us also analyze the case $n=1$. Using the hypothesis and \eqref{eq 64} we obtain
\begin{align*}
\begin{cases}
&P\psi(0)=0\Rightarrow %-\partial_x^5\psi(0)+\partial_x^3\psi(0)+\partial_x\psi(0)=0\Rightarrow 
\partial_x^5\psi(0)=0,\quad \partial_xP\psi(0)=0\partial_x^6\psi(0)=0,\quad \partial_x^2P\psi(0)=0\Rightarrow \partial_x^7\psi(0)=0,\\
&\partial_x^3P\psi(0)=0\Rightarrow %-\partial_x^8\psi(0)+\partial_x^6\psi(0)+\partial_x^4\psi(0)=0\Rightarrow 
\partial_x^8\psi(0)=0,\quad \partial_x^4P\psi(0)=0\Rightarrow %-\partial_x^9\psi(0)+\partial_x^7\psi(0)+\partial_x^5\psi(0)=0\Rightarrow 
\partial_x^9\psi(0)=0.
\end{cases}
\end{align*}
Combining this with \eqref{eq 64} we get \eqref{eq 63}, for $n=1$. Now, suppose that \eqref{eq 63} holds for some $n\geq 1$ and let us show that
\begin{align*}
\partial_x^j\psi(0)=0\ \ \forall j \in \{0,1,...,5(n+1)+4\}.
\end{align*}
By the induction hypothesis, it is sufficient to show that $\partial_x^j\psi(0)=0$ for $j=5(n+1)+r$ with $r=0,1,2,3,4$. From hypothesis $P^{n+1}\psi(0)=0$, that is,
\begin{align*}
\sum_{q=0}^{n+1}\binom{n+1}{q}\sum_{k=0}^q\binom{q}{k}(-1)^k\partial_x^{n+1+2q+2k}\psi(0)=0.
\end{align*}
Thus,
\begin{align*}
	\sum_{q=0}^{n}\binom{n+1}{q}\sum_{k=0}^q\binom{q}{k}(-1)^k\partial_x^{n+1+2q+2k}\psi(0)+\binom{n+1}{n+1}\sum_{k=0}^{n+1}\binom{n+1}{k}(-1)^k\partial_x^{n+1+2(n+1)+2k}\psi(0)=0.
\end{align*}
Note that $n+1+2q+2k\leq  5n+1$, for $0\leq k\leq q\leq n$. So, from induction hypothesis it follows that $\partial_x^{n+1+2q+2k}\psi(0)=0$ and the last equality becomes
\begin{align*}
\sum_{k=0}^{n+1}\binom{n+1}{k}(-1)^k\partial_x^{3(n+1)+2k}\psi(0)=0,
\end{align*}
or, equivalently,
\begin{align*}
\sum_{k=0}^{n}\binom{n+1}{k}(-1)^k\partial_x^{3(n+1)+2k}\psi(0)+\binom{n+1}{n+1}(-1)^{n+1}\partial_x^{5(n+1)}\psi(0)=0.
\end{align*}
But, for $0\leq k\leq n$ we have, $3(n+1)+2k\leq 5n+3$, and the induction hypothesis gives us $\partial_x^{3(n+1)+2k}\psi(0)=0$ and therefore, from the last equality we concludes $\partial_x^{5(n+1)}\psi(0)=0$. In a similar way,
\begin{align*}
\partial_x^rP^{n+1}\psi(0)=0  \implies \partial_x^{5(n+1)+r}\psi(0)=0,\ \ \ r=1,2,3,4,
\end{align*}
concluding the proof of \eqref{eq 63}, and we conclude that $\partial_x^j\psi(0)=0$, for all $j\geq 0.$  Since $\psi$ is analytic in $[-1,0]$, it follows that $\psi\equiv0$.
\end{proof}

\subsection{Reachable states} We are now in a position to prove the second main result of the article.

\begin{proof}[Proof of Theorem \ref{main2}]Let $R>2R_0$ and $u_1 \in \mathcal{R}_R$. Later on, it will be shown that $u_1$ can be written in the form
\begin{align}\label{series for u_1}
u_1(x)=\sum_{i\geq 0}c_if_i(x)+\sum_{i\geq 0}b_ig_i(x),\ \forall x \in [-1,0].
\end{align}
Assume for a moment that \eqref{series for u_1} holds with a convergence in $W^{n,\infty}(-1,0)$ for all $n\geq 0$. Then, using \eqref{P^kf_j} and \eqref{P^kg_j} we obtain
\begin{align*}
P^nu_1(x)=(-1)^n\sum_{i\geq n}c_if_{i-n}(x)+(-1)^n\sum_{i\geq n}b_ig_{i-n}(x)
\end{align*}
and
\begin{align*}
\partial_x^jP^nu_1(x)=(-1)^n\sum_{i\geq n}c_i\partial_x^jf_{i-n}(x)+(-1)^n\sum_{i\geq n}b_i\partial_x^jg_{i-n}(x)\ \ \forall j\geq 0.
\end{align*}
From \eqref{f_0}-\eqref{gj} it follows that
\begin{align*}
\partial_x^3P^nu_1(0)=(-1)^nc_n\partial_x^3f_0(0)+(-1)^n\sum_{i> n}c_i\partial_x^3f_{i-n}(0)+(-1)^n\sum_{i\geq n}b_i\partial_x^3g_{i-n}(0)=(-1)^nc_n
\end{align*}
and
\begin{align*}
\partial_x^4P^nu_1(0)=(-1)^n\sum_{i\geq n}c_i\partial_x^4f_{i-n}(0)+(-1)^nb_n\partial_x^4g_0(0)+(-1)^n\sum_{i> n}b_i\partial_x^4g_{i-n}(0)=(-1)^nb_n.
\end{align*}
This leads us to define
\begin{align}\label{c_n and b_n}
c_n=(-1)^n\partial_x^3P^nu_1(0)\text{ \ \ and \ \ }b_n=(-1)^n\partial_x^4P^nu_1(0),\ \forall\ n \geq 0.
\end{align}

\vspace{0.1cm}
\noindent{\textbf{Claim.}} There exist $r \in (R_0,R)$ and a constant  $K=K(r)>0$ such that
\begin{align*}
\left|\partial_x^nu_1(x)\right|\leq K\frac{n!}{r^n},\ \forall\ x \in [-1,0].
\end{align*}

Indeed, since $R>2R_0$ we can write $R=2R_0+\alpha$ with $\alpha>0$. Then $R>R_0+R_0+\frac{\alpha}{2}$, taking $r=R_0+\frac{\alpha}{2}$ we have $r \in (R_0,R)$ and $R_0+r<R$. Consequently $\overline{D(w,r)}\subset \overline{D(0,R_0+r)}\subset D(0,R),$ for all  $w \in \overline{D(0,R_0)}$. Define $K:=\max\left\{|z(w)|;\ w \in \overline{D(0,R_0+r)}\right\},$ where $z \in H(D(0,R))$ is such that $z|_{[-1,0]}=u_1$. So, under the hypothesis of $u_1\in\mathcal{R}$, follows that 
\begin{align*}
	\left|\partial_x^nu_1(x)\right|\leq K\frac{n!}{r^n},\ \forall\ x \in [-1,0],
\end{align*}
as desired, showing the claim.\color{black}

\vspace{0.1cm} 

Using Lemma \ref{lemma A3}, with $p=\infty$, the claim and \eqref{inequality} we get that
\begin{align*}
|c_n|\leq \sum_{i=0}^n\|P^i\partial_x^3u_1\|_\infty&\leq \frac{3^{n+1}}{2}\|\partial_x^3u_1\|_{5n,\infty}
%\leq \frac{3^{n+1}}{2}\sum_{i=0}^{5n}\|\partial_x^{i+3}u_1\|_{\infty}\leq \frac{3^{n+1}}{2}\sum_{i=0}^{5n}K\frac{(i+3)!}{r^{i+3}}\\
%&\leq \frac{3^{n+1}}{2r^3}K(5n+3)!\sum_{i=0}^{5n}\frac{1}{r^{i}}\\
%&=\frac{3^{n+1}}{2r^3}K(5n+3)!\frac{\left(\frac{1}{r}\right)^{5n+1}-1}{\frac{1}{r}-1}\\
\leq \frac{3^{n+1}}{2r^3}K\cdot 2^{5n}\cdot 2^3\cdot (5n)!\cdot 3!\cdot \frac{\frac{1}{r}}{\frac{1}{r}-1}\cdot \frac{1}{r^{5n}}.
\end{align*}
and analogously
\begin{align*}
|b_n|\leq \frac{3^{n+1}}{2r^4}K\cdot 2^{5n}\cdot 2^4\cdot (5n)!\cdot 4!\cdot \frac{\frac{1}{r}}{\frac{1}{r}-1}\cdot \frac{1}{r^{5n}}.
\end{align*}
Therefore
\begin{align*}
|c_n|,|b_n|&\leq K'\left(\frac{3\cdot 2^{5}}{r^{5}}\right)^n(5n)!,\ \forall\ n\geq 0,
\end{align*}
for some positive constant $K'>0$. 

Define $H=\frac{3\cdot 2^5}{r^5}$ and observe that $r>R_0$ implies that $He^{e^{-1}}<\frac{1}{2}$.  Choose $\tilde{H} \in \left(He^{e^{-1}},\frac{1}{2}\right)$. Then, from Proposition \ref{prop 3.2}, there exist functions $\tilde{f},\tilde{g} \in C^\infty(\mathbb{R})$ such that
$$
\begin{aligned}
&\tilde{f}^{(n)}(0)=c_n,\ \ \ \tilde{g}^{(n)}(0)=b_n,\quad \text{and} \quad \left|\tilde{f}^{(n)}(t)\right|,\left|\tilde{g}^{(n)}(t)\right|\leq K'\tilde{H}^n(5n)!\ \ \forall t \in \mathbb{R},
\end{aligned}
$$
for every $n\geq 0$. Define $f(t)=\tilde{f}(t-T)$ and $g(t)=\tilde{g}(t-T)$. Then $f,g \in C^\infty(\mathbb{R})$ with
$$
f^{(n)}(t)=\tilde{f}^{(n)}(t-T)\text{ \ \ and \ \ }g^{(n)}(t)=\tilde{g}^{(n)}(t-T)\ \ \forall n\geq 0,\ \forall t \in \mathbb{R}.
$$
Moreover, $f^{(n)}(T)=c_n$, $g^{(n)}(T)=b_n$, and
$$
\left|f^{(n)}(t)\right|,\left|g^{(n)}(t)\right|\leq K'\tilde{H}^n(5n)!\ \ \ \forall n\geq 0,\ \forall t \in \mathbb{R}.
$$
From the last inequality, we have
\begin{align}\label{eq 43}
\left|f^{(n)}(t)\right|,\left|g^{(n)}(t)\right|\leq K'\frac{(5n)!}{R_3^{5n}}\ \  \forall n\geq 0,\ \forall t \in \mathbb{R},
\end{align}
where $R_3=\tilde{H}^{-\frac{1}{5}}>2^{\frac{1}{5}}$. Note that \eqref{eq 43} implies that $f,g \in G^5([0,T])$. Indeed, from Stirling's formula we have
\begin{align*}
	(5n)!\sim 5^{5n+\frac{1}{2}}\left(\sqrt{2\pi n}\right)^{-4}n!^5=\frac{5^\frac{1}{2}(2\pi)^{-2}}{n^2}5^{5n}n!^5.
\end{align*}
Hence, for some positive constant $K''>0$ we have
\begin{align}\label{eq 66}
\left|f^{(n)}(t)\right|,\left|g^{(n)}(t)\right|\leq K'\frac{(5n)!}{R_3^{5n}}\leq K''\frac{5^{5n}}{R_3^{5n}}n!^5=K''\frac{n!^5}{\left(5^{-5}R_3^5\right)^n}.
\end{align}

Pick any $\tau \in (0,T)$ and let
\begin{align*}
\beta(t)=1-\phi_2\left(\frac{t-\tau}{T-\tau}\right),\ \ \ t \in [0,T].
\end{align*}
Observe that
\begin{align*}
	\beta^{(i)}(t)=\left(-\frac{1}{T-\tau}\right)^i\phi_2^{(i)}\left(\frac{t-\tau}{T-\tau}\right),\ i\geq 1.
\end{align*}
By definition we have $\supp \phi_2\subset \subset (-\infty,1)$ so $\beta(T)=1$ and $\beta^{(i)}(T)=0,$ for all $i\geq 1.$ Define
\begin{align}\label{y and z}
y(t)=f(t)\beta(t)\text{ \ \ and \ \ }z(t)=g(t)\beta(t),\ \ t \in [0,T].
\end{align}
From the Leibniz rule, we have
\begin{align*}
y^{(n)}(t)=\sum_{j=0}^n\binom{n}{j}f^{(j)}(t)\beta^{(n-j)}(t)\text{ \ \ \ and \ \ \ }z^{(n)}(t)=\sum_{j=0}^n\binom{n}{j}g^{(j)}(t)\beta^{(n-j)}(t).
\end{align*}
Then
\begin{align}\label{y^n in T}
y^{(n)}(T)=\sum_{j=0}^n\binom{n}{j}f^{(j)}(T)\beta^{(n-j)}(T)=f^{(n)}(T)\beta(T)=c_n
\end{align}
and
\begin{align}\label{z^n in T}
z^{(n)}(T)=\sum_{j=0}^n\binom{n}{j}g^{(j)}(T)\beta^{(n-j)}(T)=g^{(n)}(T)\beta(T)=b_n.
\end{align}
On the other hand, by definition $\phi_2\equiv 1$ in $(-\infty,0]$ so that $\beta \equiv 0$ in $(-\infty,\tau)$. Consequently $y^{(n)}\equiv z^{(n)}\equiv 0$ in $(-\infty,\tau)$ for any $n\geq 0$. In particular
\begin{align}\label{y^n,z_n in 0}
y^{(n)}(0)=z^{(n)}(0)=0,\ \ \forall n\geq 0.
\end{align}
Since $\beta \in G^2([0,T])$ and $f,g \in G^5([0,T])$ we have $y,z \in G^5([0,T])$. Moreover, from \cite[Lemma 3.7]{Reachable states for heat equation}, the same constant ``$R$~'' of $f$ and $g$ in the definition of $G^5([0,T])$ works for $y$ and $z$. Hence, from \eqref{eq 66}, there exists $K'''>0$ such that
\begin{align*}
\left|y^{(n)}(t)\right|,\left|z^{(n)}(t)\right|\leq K''' \frac{n!^5}{\left(5^{-5}R_3^5\right)^n},\ \ \ \forall n\geq 0,\ \forall t \in [0,T].
\end{align*}
Using Stirling's formula we obtain $(5n)!\sim 5^{5n+\frac{1}{2}}\left(\sqrt{2\pi n}\right)^{-4}n!^5$, and so
\begin{align*}
n!^5\sim 5^{-\left(5n+\frac{1}{2}\right)}\left(\sqrt{2\pi n}\right)^{4}(5n)!=\frac{(2\pi)^2}{5^\frac{1}{2}}n^2\frac{(5n)!}{5^{5n}}.
\end{align*}
Then, there exists a constant $\overline{K}>0$ such that
\begin{align*}
\left|y^{(n)}(t)\right|,\left|z^{(n)}(t)\right|\leq K'''\frac{1}{\left(5^{-5}R_3^5\right)^n}\overline{K}n^2\frac{(5n)!}{5^{5n}}=K'''\overline{K}\cdot \frac{n^2(5n)!}{R_3^{5n}}.
\end{align*}
Note that, since $R_3>2^\frac{1}{5}$ we can pick $\rho \in \left(1,R_32^{-\frac{1}{5}}\right)$, so that the sequence $\left(\frac{n^2}{\rho^{5n}}\right)_{n\in \mathbb{N}}$ is bounded. Indeed, using the L'Hospital rule we see that
\begin{align*}
\lim_{x\rightarrow\infty}\frac{x^2}{\rho^{5x}}=\lim_{x\rightarrow\infty}\frac{2x}{(5\ln\rho)\rho^{5x}}=\lim_{x\rightarrow\infty}\frac{2}{(5\ln\rho)^2\rho^{5x}}=0.
\end{align*}
Hence, there exists $\overline{K}'>0$ such that $n^2\leq \overline{K}'\rho^{5n}$, and therefore
\begin{align*}
	\left|y^{(n)}(t)\right|,\left|z^{(n)}(t)\right|\leq K'''\overline{K}\cdot \overline{K}'\cdot \frac{\rho^{5n}(5n)!}{R_3^{5n}}=K'''\overline{K}\cdot \overline{K}'\cdot \frac{(5n)!}{\left(R_3\rho^{-1}\right)^{5n}}.
\end{align*}
Defining $K''''=K'''\overline{K}\cdot \overline{K}'$ and $R_3'=\frac{R_3}{\rho}$ we have $R_3'>2^{\frac{1}{5}}$ and
\begin{align}\label{eq 68}
\left|y^{(n)}(t)\right|,\left|z^{(n)}(t)\right|\leq K''''\frac{(5n)!}{\left(R_3'\right)^{5n}},\ \ \forall n\geq 0, \forall t \in [0,T].
\end{align}
%Observe that, since $\rho>1$ we have $\frac{1}{R_3}<\frac{1}{R_3'}$. Then. \eqref{eq 43} and \eqref{eq 67} hold with the same constant $R_3'$.

Let $u$ be as in \eqref{flatness solution} corresponding to $y$ and $z$ given in \eqref{y and z}. From \eqref{eq 68} and by Proposition \ref{flatness for s=5} we have that $y \in G^{1,5}([-1,0]\times [0,T])$ and it solves \eqref{flatness system}. Furthermore, \eqref{y^n,z_n in 0} gives us
\begin{align*}
u(x,0)=\sum_{j\geq 0}f_j(x)y^{(j)}(0)+\sum_{j\geq 0}g_j(x)z^{(j)}(0)=0.
\end{align*}
Setting $h_1=u(-1,t)$ and $h_2=u_x(-1,t)$, we get $h_1,h_2 \in G^5([0,T])$ and therefore $u$ solves \eqref{kawahara control system} with $h_1$ and $h_2$ as control inputs and  $y_0=0$ as initial data. From the proof of Proposition \ref{flatness for s=5} we know that for all $n,m\in \mathbb{N}$ the sequence of the series
$$
\sum_{j\geq 0}\partial_t^m\partial_x^n\left(f_j(x)y^{(j)}(t)+g_j(x)z^{(j)}(t)\right)
$$
converges uniformly on $[-1,0]\times [0,T]$ to $\partial_t^m\partial_x^nu$ and consequently, for all $n,i\geq 0$,
\begin{align*}
P^nu(x,t)&=\sum_{j\geq 0}P^nf_j(x)y^{(j)}(t)+\sum_{j\geq 0}P^ng_j(x)z^{(j)}(t)\\
&=(-1)^n\sum_{j\geq n}f_{j-n}(x)y^{(j)}(t)+(-1)^n\sum_{j\geq n}g_{j-n}(x)z^{(j)}(t)
\end{align*}
and
\begin{align*}
\partial_x^iP^nu(x,t)&=(-1)^n\sum_{j\geq n}\partial_x^if_{j-n}(x)y^{(j)}(t)+(-1)^n\sum_{j\geq n}\partial_x^ig_{j-n}(x)z^{(j)}(t),
\end{align*}
for every $(x,t)\in [-1,0]\times [0,T]$. In particular,  \eqref{y^n in T} and \eqref{z^n in T} gives us that
\begin{align*}
\partial_x^iP^nu(x,T)=(-1)^n\sum_{j\geq n}c_j\partial_x^if_{j-n}(x)+(-1)^n\sum_{j\geq n}b_j\partial_x^ig_{j-n}(x)\ \ \forall i,n\geq0,\ \forall x \in [-1,0].
\end{align*}
Thus, \eqref{f_0}-\eqref{gj} provide us
\begin{equation*}
\begin{cases}
\begin{aligned}
P^nu(0,T)&=\partial_xP^nu(0,T)=\partial_x^2P^nu(0,T)=0,\\
\partial_x^3P^nu(0,T)&=(-1)^n\sum_{j\geq n}c_j\partial_x^3f_{j-n}(0)+(-1)^n\sum_{j\geq n}b_j\partial_x^3g_{j-n}(0)=(-1)^nc_n,\\
\partial_x^4P^nu(0,T)&=(-1)^n\sum_{j\geq n}c_j\partial_x^4f_{j-n}(0)+(-1)^n\sum_{j\geq n}b_j\partial_x^4g_{j-n}(0)=(-1)^nb_n,
\end{aligned}
\end{cases}
\end{equation*}
and therefore, using \eqref{c_n and b_n},
\begin{equation}\label{eq 73}
\begin{cases}
\begin{aligned}
&\partial_x^jP^nu(0,T)=0,\ j=0,1,2,\\
&\partial_x^3P^nu(0,T)=\partial_x^3P^nu_1(0),\quad \partial_x^4P^nu(0,T)=\partial_x^4P^nu_1(0).
\end{aligned}
\end{cases}
\end{equation}
Define $\psi \in G^1([-1,0])$ by $\psi(x)=u(x,T)-u_1(x),$ for all $x \in [-1,0]$, and using the fact that $u_1 \in \mathcal{R}_R$ together with \eqref{eq 73}, holds that $\partial_x^jP^n\psi(0)=0,$ for $j=0,1,2,3,4.$  From Lemma \ref{lemma 3.1} it follows that $\psi \equiv 0$, that is, $u(x,T)=u_1(x)$, which concludes the proof.
\end{proof}

%\section{Further comments and open issues}

\subsection*{Acknowledgments.}  This work is part of the Ph.D. thesis of da Silva at the Department of Mathematics of the Universidade Federal de Pernambuco.

\end{document}